%% file: MimickingODE.tex
\documentclass[11pt,a4paper]{article}
\usepackage[margin=1.05in]{geometry}
\usepackage{graphicx}
\usepackage{amsmath,amssymb,amsthm,mathtools}
\usepackage{bm,hyperref}
\usepackage{tikz} \usetikzlibrary{arrows.meta,positioning,fit}
\theoremstyle{plain}

\newtheorem{theorem}{Theorem}[section]
\newtheorem{proposition}[theorem]{Proposition}
\newtheorem{lemma}[theorem]{Lemma}
\newtheorem{corollary}[theorem]{Corollary}

\theoremstyle{definition}
\newtheorem{definition}[theorem]{Definition}
\newtheorem{assumption}{Assumption}

\theoremstyle{remark}
\newtheorem{remark}[theorem]{Remark}

\theoremstyle{example}

\newcommand{\R}{\mathbb{R}}
\newcommand{\E}{\mathbb{E}}

\newcommand{\Law}{\operatorname{Law}}
\newcommand{\dd}{\mathrm{d}}
\newcommand{\Div}{\nabla\!\cdot\!}
\newcommand{\Leb}{\mathcal{L}^d}
\newcommand{\Cov}{\operatorname{Cov}}

\newcommand{\loc}{\mathrm{loc}}
\newcommand{\Ent}{\operatorname{Ent}}
\newcommand{\Fish}{\mathcal{I}}
\newcommand{\opnorm}[1]{\lVert #1 \rVert_{\mathrm{op}}}
\newcommand{\BV}{\mathrm{BV}}
\DeclareMathOperator{\supp}{supp}

\title{Mimicking diffusion processes with differential equations}
\author{Rama Cont\footnote{The author thanks Luigi Ambrosio and Gabriel Peyr\'e, whose insightful lectures at the Spring School on Mathematics of Random Systems (Pisa, 2025)  inspired this work, and RenYuan Xu for enlightening discussions on generative models.}}
\date{Mathematical Institute, University of Oxford}

\begin{document}
\maketitle
\vspace{-2.2\baselineskip}

\begin{abstract}
\noindent
The probability-flow ordinary differential equation (PF-ODE) associated with a diffusion process is widely used in score-based generative modeling as a deterministic sampler that reproduces the marginal distributions of the diffusion. The validity of this marginal-matching property  depends on the well-posedness of an ordinary differential equation whose velocity field is constructed from the score function of the diffusion. 
We examine the precise mathematical relation between a diffusion process, the Fokker-Planck equation and the associated probability-flow ODE (PF-ODE) under weak regularity assumptions on the drift and score function.

We first establish existence and uniqueness of the marginal density flow as a solution of the Fokker--Planck equation under minimal assumptions on the diffusion coefficients. We then study the corresponding Lagrangian problem using the DiPerna--Lions--Ambrosio theory of regular Lagrangian flows. We prove existence, uniqueness and stability of the PF-ODE flow, and show that it transports the initial distribution onto the diffusion marginals, under Sobolev or bounded-variation regularity of the score together with one-sided bounds on the divergence of the probability-flow velocity.
We identify sufficient conditions for the required regularity in diffusion models relevant for applications.

Our analysis underlines a fundamental distinction between Eulerian and Lagrangian descriptions. We construct a counterexample in which the Fokker--Planck equation has a unique density flow while the associated PF-ODE fails to admit a regular Lagrangian flow from the initial time, demonstrating that uniqueness of the density evolution does not in general imply the existence of a deterministic probability-flow representation. Finally, we derive stability estimates for probability-flow trajectories under learned score approximations. We discuss the implications of our findings for the training and deployment of generative diffusion models.
\end{abstract}
DiPerna–Lions theory;  Kolmogorov equations; Linear transport equations;  generative models. 
\newpage
\tableofcontents
\newpage
\section{Mimicking diffusions with differential equations}\label{sec.intro}

\subsection{Generative diffusion models and continuous probability flows}\label{sec.generative}
Score-based diffusion models
\cite{SohlDickstein2015,Song2021mle,Song2021,yang2023} have become state-of-the-art generative models for image and video generation \cite{kim2024,shen2026}.
Score-based diffusion models  progressively add noise to a data set through a forward diffusion process, learn the score function $s(x,t)=\nabla\log p_t(x)$ associated with marginals, and then generate samples by applying time-reversed dynamics : either via a time-reversed stochastic differential equation or via a deterministic probability-flow ODE which transports a Gaussian distribution back onto the data distribution \cite{peyre2026}.

Score-based diffusion models use two ingredients \cite{Song2021,yang2023}: a stochastic differential equation (SDE)--the {\it forward diffusion process}-- to transform a complex data distribution $p_0$ to a reference (Gaussian) distribution $p_T$ by slowly injecting noise, and a corresponding time-reversed SDE which generates the reverse probability flow  $q_t=p_{T-t}$, thus transforming the reference distribution back into the data distribution. Crucially, the time-reversed SDE depends only on the score function $\nabla \log p_t(x)$ of the forward diffusion \cite{anderson1982,haussmann1986}. By estimating the score function with a neural network, one can then  generate samples from the data distribution by simulating the time-reversed SDE.

One of the features underlying the computational efficiency of these models is the idea, used by Song et al. \cite{Song2021}, that one may sample from the  marginal densities $(p_t)_{t\in [0,T]}$ of a diffusion process by first sampling from the initial distribution $p_0$ then transporting the initial sample via an  ordinary differential equation (ODE), called the {\it probability-flow ODE} (PF-ODE).

Indeed, consider a  diffusion process in $\mathbb{R}^d$ driven by a $d-$dimensional Brownian motion $W$:
$$ dX_t= f(X_t,t) dt + \sigma(t) dW_t,\qquad X_0\sim \mu_0 $$
with (scalar) diffusion coefficient $\sigma:[0,T]\mapsto (0,\infty)$ and initial distribution $\mu_0\in {\cal P}(\mathbb{R}^d)$. Assume $X_t$  has a strictly positive density $p_t(x)$ on $\mathbb{R}^d$ for $t>0$. 
The density $p_t(x)$ solves  the Fokker--Planck / Kolmogorov forward equation
on $\R^d\times(0,T)$:
\begin{equation}\label{eq.fokkerplanck}
  \partial_t p_t(x)+\Div(f(.,t) p_t)(x)-\tfrac12\sigma(t)^2\Delta p_t(x)=0.
\end{equation}
Song et al \cite{Song2021} considered the  ordinary differential equation:
\begin{equation}\label{eq.pfode}
  \tfrac{\dd}{\dd t}Z(t,x)=v\big(Z(t,x),t\big),\quad Z(0,x)=x,\quad{\rm where}\quad v(x,t):=f(x,t)-\tfrac12\sigma(t)^2\nabla\log p_t(x).
\end{equation}
Assuming this ODE is well-posed, if we initialize it with a random initial condition $Z_0$ with density $p_0$ then the density $p_t$ of $Z(t,Z_0)$  evolves according to the {\it continuity equation}
\begin{eqnarray}\label{eq.cont}
 \partial_t p_t +\Div(p_t\ v)=0
\end{eqnarray}
which describes the 'probability flow' associated with the ODE \eqref{eq.pfode}.

Using $\Delta p_t=\Div\ (p_t\nabla\log p_t)$, we observe that the continuity equation \eqref{eq.cont} is formally identical to the Fokker-Planck equation \eqref{eq.fokkerplanck}. 
So, if the ODE \eqref{eq.pfode} is well-posed and the PDEs \eqref{eq.fokkerplanck} and \eqref{eq.cont} have unique solutions, then by  initializing the ODE  with a random initial condition $X_0\sim \mu_0$ we obtain a solution $Z_t=Z(t,X_0)$ which has the {\it same distribution} as $X_t$: $${\rm Law}(Z(t,X_0) )={\rm Law}(X_t).$$
Thus the ODE \ref{eq.pfode} with initial condition $X_0$ {\it mimicks} the flow of marginal distributions of the diffusion process $X$: \eqref{eq.pfode} is called the  'probability flow ODE' (\emph{PF-ODE}) associated with the diffusion $X$.

This  {\it probability-flow ODE} (PF-ODE) may thus be used to generate samples from the distribution $(p_t)$ of $X_t$ (or the reverse flow $q_t=p_{T-t}$) {\it without} additional random simulation,  by sampling $Z_{t_0}\sim p_{t_0}$ then transporting it along the PF-ODE \eqref{eq.pfode} either forwards or backwards.
The probability-flow ODE has become the  sampler of choice in diffusion models, as it easier to simulate and  leads to exact likelihood evaluation and deterministic decoding \cite{Song2021mle}. 

The idea of "mimicking" the marginal probability flow of  a stochastic process with a simpler, analytically tractable  process goes back to Gy\"ongy
\cite{gyongy1986}, who proposed a method for mimicking the marginal flow of an Ito process with a Markovian-type SDE. Here the idea is taken one step further, namely mimicking the marginal flow of a diffusion process with a {\it deterministic} differential equation.

The validity of  this marginal-matching property, often assumed in the literature on generative models, depends on the well-posedness of the nonlinear differential equation \eqref{eq.pfode}. Even when $f$ is Lipschitz-continuous and $\sigma$ bounded, the vector field $v$ may not always satisfy the Cauchy-Lipschitz assumptions for well-posedness of the ODE \eqref{eq.pfode}, as it involves regularity of the score function $s(x,t)=\nabla\log p_t$. 
The relations between the PDE, the probability flow and the ODE thus need to be qualified and only hold under certain conditions. 

\subsection{Contributions}

Inspired by the questions arising in the context of generative diffusion models, we examine the precise mathematical relation between the diffusion process $X$, the associated Fokker-Planck equation \eqref{eq.fokkerplanck}  and the probability-flow ODE (PF-ODE) \eqref{eq.pfode}, under weak regularity assumptions on the drift and the score function.
We distinguish  two complementary viewpoints:
\begin{itemize}
\item the \emph{Eulerian} representation: Is the density flow $(p_t)$ of $X_t$ the unique solution of
the Fokker-Planck equation \eqref{eq.fokkerplanck} in a natural class? This is a question about the linear parabolic PDE \eqref{eq.fokkerplanck} with possibly irregular coefficients, which we tackle in Section \ref{sec.euler}.
\item the  \emph{Lagrangian} representation: Does the ODE \eqref{eq.pfode} generate a well-defined flow
$Z$, and does this flow transport $p_0$ onto $p_T$? This is a question about the first-order ODE \eqref{eq.pfode} and the associated transport equation \eqref{eq.cont}. We tackle this question in Section \ref{sec.lagrange} using the DiPerna--Lions theory \cite{DL1989,Amb2004} of regular Lagrangian flows and its extension by Ambrosio \cite{Amb2004,Ambrosio2008}. 
\end{itemize}
In particular we examine whether, as implicitly assumed in many studies on diffusion models, these representations are interchangeable, or equivalent, ways of generating a desired probability flow.
This requires to understand   in detail the mathematical relation between the SDE, its marginal flow, the PF-ODE and the continuity equation. 

Our analysis combines two methodologies: the well-posedness of Fokker–Planck equations with irregular coefficients \cite{Fig2008,LBL2008,LBL2019} and the DiPerna–Lions–Ambrosio theory of regular Lagrangian flows \cite{DL1989,Amb2004,Ambrosio2008}. This allows us to formulate the probability-flow ODE under minimal regularity assumptions and to identify the role of each assumption.
The counterexample in Section 
\ref{sec:counterexample} provides further insight into the role of various assumptions.

Replacing the score function $\nabla\log p_t$ by a learned (neural network) approximation $s_\theta$ gives an estimator $v_\theta$.
A question of interest is to quantify  the impact $\|Z_\theta-Z \|$ of this estimation error
 for the flow. Such quantitative error estimates are the focus of Section~\ref{sec:stability}. Theorem~\ref{thm:cdl} studies 
the impact of the estimation error on the flow and Proposition~\ref{prop:transfer} examines the relation with  the
score-matching objective function. Importantly, our error estimates require conditions on the design of the estimator but not on the loss function or the training procedure. These results complement recent work by Han et al. \cite{hanxu2024} on error estimates for score-based diffusion models.

\subsection{Relation with previous research} 

Recent theoretical work \cite{BrigatiPedrotti2025,iske2026,mooney2025global,Stephanovitch2026Lipschitz,Stephanovitch2025ScoreRegularity} has established a rigorous understanding of probability-flow dynamics in the Cauchy-Lipschitz regime where the coefficients and score satisfy strong regularity conditions. 

Mooney, Wang, Xin, and Yu \cite{mooney2025global} derive  Lipschitz estimates for diffusion scores for a broad class of initial distribution and use local Lipschitz control to prove global well-posedness and convergence of score-based generative dynamics. Their analysis also identifies  deterioration of score-Lipschitz bounds for non-log-concave distributions and an $O(1/t)$ regularization rate for nonsmooth measures supported on compact smooth manifolds with boundary. St\'ephanovitch \cite{Stephanovitch2026Lipschitz} obtains  Lipschitz and one-sided Lipschitz estimates for diffusion and flow-matching vector fields under weak log-concavity assumptions, leading to  Lipschitz transports and quantitative sampling rates. Iske and Sch\"onlieb \cite{iske2026}  show that when the score has a uniformly bounded Hessian the probability-flow ODE generates a bi-Lipschitz diffeomorphism, and use this structure to study the expressivity of bi-Lipschitz normalizing flows. Brigati and Pedrotti \cite{BrigatiPedrotti2025} derive explicit log-Hessian bounds for log-Lipschitz perturbations of strongly log-concave measures, yielding  sufficient score-regularity  for diffusion dynamics, and show that arbitrarily strong tail decay alone does {\it not} guarantee uniform log-Hessian control. 

These results establish classical  well-posedness when quantitative score regularity is available. Our analysis is complementary:  we study probability-flow dynamics {\it beyond} the classical Lipschitz regime. Using regular Lagrangian-flow theory, we prove well-posedness under Sobolev and BV regularity, distinguish Eulerian density evolution from Lagrangian transport, and exhibit examples in which the density evolution is well defined although no regular Lagrangian probability flow exists from the initial time.

A complementary line of work \cite{hanxu2024,huang2025convergence} studies the statistical and computational aspects of score-based generative models, assuming that the score can be estimated with sufficient accuracy. Han, Razaviyayn and Xu \cite{hanxu2024} establish optimization and generalization guarantees for gradient-descent training of neural-network score estimators, proving $L^2$
 error bounds for learned scores. Huang, Huang, and Lin \cite{huang2025convergence} analyze the effect of score approximation and numerical discretization errors on probability-flow ODE sampling, deriving quantitative convergence guarantees in total variation under smoothness assumptions on the learned score. 
 Chen, Vanden-Eijnden and Xu \cite{chen2025} use the marginal-matching representation to optimize interpolation schedules by improving the Lipschitz conditioning of the resulting drift.
 
 These results quantify how statistical estimation and computational errors propagate through a  probability-flow dynamics, assumed to be well-posed. By contrast, our analysis addresses the underlying analytical question of when the probability-flow vector field itself generates a meaningful Lagrangian transport. Our stability results for learned score fields in Section \ref{sec:stability} are therefore complementary to these error analyses and go beyond the classical Cauchy-Lipschitz assumptions.


\subsection{Outline}

Section \ref{sec.setting} introduces the forward diffusion, the Fokker–Planck equation and the probability-flow ODE, together with the assumptions used in various results. Section \ref{sec.euler} treats the Eulerian problem: Theorem \ref{thm:eulerian} establishes existence and uniqueness of the density flow under weak assumptions on the drift — requiring only a one-sided bound on the scalar divergence $\nabla\cdot{}f$ and no regularity of the score — and identifies it with the marginal flow of the SDE. Section \ref{sec.lagrange} turns to the Lagrangian problem: using the Ambrosio–DiPerna–Lions theory of regular Lagrangian flows, Theorem \ref{thm:lagrangian} shows that under the additional assumption (S) on the score, the PF-ODE \eqref{eq:pfode} generates a unique regular Lagrangian flow transporting $p_0$ onto $p_t$, and Corollary \ref{cor:reverse} identifies the stronger two-sided divergence bounds required for time reversal and invertibility. Section \ref{sec:verify} provides verifiable sufficient conditions for the regularity of the score function: strongly log-concave initial data with linear drift (Proposition \ref{prop:linear}), compactly supported — possibly singular — initial laws after early stopping (Proposition \ref{prop:earlystop}), and nonlinear gradient drifts under a confinement condition (Proposition 5.4). 

Section \ref{sec:counterexample} presents a counterexample for which the density flow is uniquely determined by Theorem \ref{thm:eulerian} yet the PF-ODE admits no regular Lagrangian flow from time zero and the continuity equation fails to be solvable for a range of initial conditions. The mechanism is a $1/t$ blow-up of the score outside the initial support, which separates assumption (S4) from its Fisher-information relaxation (S4'). Section \ref{sec:stability} studies the deterministic sampler with a learned score: Theorem \ref{thm:cdl} gives a quantitative log-Lipschitz stability estimate for the flow in terms of the $L^1$ velocity error, Proposition  \ref{prop:transfer} relates this error to the score-matching loss through an inverse-density factor, and Proposition \ref{prop:gronwall} gives a linear-rate alternative for Lipschitz scores. Section \ref{sec.discussion} discusses implications  of our results for generative modelling. 

\input{PFODE}

\input{MimickingODE.bbl}
\newpage
\appendix
\section{Proof of Lemma \ref{lem:fisher}} \label{app.fisher}
\input{appendices}

\end{document}

%% file: PFODE.tex
\section{Definitions and assumptions} \label{sec.setting}
Let $T>0$, $d\ge1$, and $\Leb$ be the Lebesgue measure on $\R^d$. We consider a diffusion process
\begin{equation}\label{eq:sde}
  \dd X_t=f(X_t,t)\,\dd t+\sigma(t)\,\dd W_t,\qquad t\in[0,T],\qquad X_0\sim \mu_0\in {\cal P}(\mathbb{R}^d),
\end{equation}
with $\sigma:[0,T]\to(0,\infty)$ scalar and no spatial dependence. 
When the SDE \eqref{eq:sde} has a unique solution, we will denote $p_t(x)$  the density of $X_t$.
The density $p_t(x)$ solves  the Fokker--Planck / Kolmogorov forward equation
on $\R^d\times(0,T)$:
\begin{equation}\label{eq:fp}
  \partial_t p_t(x)+\Div(f p_t)(x)-\tfrac12\sigma(t)^2\Delta p_t(x)=0.
\end{equation}
We associate with this diffusion process a 'probability flow differential equation' (\emph{PFODE}) \cite{Song2021}
\begin{eqnarray}\label{eq:pfode}
  \tfrac{\dd}{\dd s}Z(t,x)=v\big(Z(t,x),t\big)&\qquad Z(0,x)=x,
  \\
{\rm where}&\qquad  \boxed{\;v(x,t):=f(x,t)-\tfrac12\sigma(t)^2\nabla\log p_s(x).\;}\label{eq:vectorfield}
\end{eqnarray}
The {\it continuity} equation associated with the ODE \eqref{eq:pfode} is the linear transport equation
\begin{eqnarray}\label{eq:cont}
&\partial_t \rho+\Div(\rho v)=0,
\end{eqnarray}
As noted in Section \ref{sec.intro}, the interest in the particular choice of the vector field \eqref{eq:vectorfield} stems from the observation that, under certain assumptions (to be specified) the continuity equation \eqref{eq:cont} formally coincides with \eqref{eq:fp}.

Our goal is to study the mathematical relation between the {\it Eulerian} description of the marginal probability flow $(p_t)_{t\in [0,T]}$ of the diffusion \eqref{eq:sde} through its Fokker–Planck equation \eqref{eq:fp} and the {\it Lagrangian} description of the probability flow provided by the  ODE \eqref{eq:pfode}.

For the Eulerian problem,   under uniform ellipticity of the second-order term, uniqueness follows from a direct energy estimate in the class where
$\nabla p\in L^2$. The essential point is that testing the equation against the difference 
of two solutions $p^1,p^2$ produces $\int f\cdot\nabla(|p^1-p^2|^2)$, which after one integration by parts
involves $\Div f$ \emph{alone}: the full gradient $\nabla f$ and the score never appear. This places the
Eulerian problem inside the classical variational framework for parabolic equations with
irregular data \cite{LBL2019}.

For the Lagrangian problem, since
the vector field $v$ in \eqref{eq:vectorfield},  contains  the score function $s(x,t)=\nabla\log p_t$,
assumptions on the regularity of the score function  are  unavoidable. We will use Ambrosio's BV extension \cite{Amb2004,Ambrosio2008} to the DiPerna--Lions theory \cite{DL1989}, which will allow us to impose weak regularity assumptions on the coefficients involved.

We will thus consider various sets of assumptions in the following sections.
\begin{assumption}[Diffusion coefficient--- (E1), (E2)]\label{ass:E}\  \\
\begin{enumerate}
\item[(E1)] $\sigma\in L^\infty(0,T)$ and $0<\epsilon\le\sigma(s)\le\epsilon^{-1}$ for a.e.\ $s$.
\item[(E2)]  $\sigma\in C^{\alpha/2}([0,T])$ for some $\alpha\in(0,1)$.
\end{enumerate}
\end{assumption}
Theorem~\ref{thm:eulerian} uses only \textup{(E1)}; \textup{(E2)} is used in
Section~\ref{sec:verify} for parabolic regularity.
\begin{assumption}[Drift  (D)]\label{ass:D}\ \\
\begin{enumerate}
\item[(D0)] $f\in L^2_\loc\big([0,T]\times\R^d;\R^d\big)$;
\item[(D1)] $f\in L^1\big([0,T];\BV_\loc(\R^d;\R^d)\big)$.\item[(D2)] $\Div f\in L^1\big([0,T];L^1_\loc(\R^d)\big)$
  and $[\Div f]^-\in L^1\big([0,T];L^\infty(\R^d)\big)$;
\item[(D3)] $\dfrac{f}{1+|x|}\in L^1\big([0,T];L^1+L^\infty(\R^d)\big)$;
\end{enumerate}
We shall also need, for Theorem~\ref{thm:eulerian} and again in
Section~\ref{sec:stability}, a pointwise bound on $f$:
\begin{enumerate}
\item[(D4)] $|f(x,t)|\le C_f(t)\,(1+|x|)$ for a.e.\ $(x,t)$, with $C_f\in L^1(0,T)$.
\end{enumerate}
\textup{(D4)} implies \textup{(D3)} but not conversely: in \textup{(D3)} the component
$f_1$ with $\tfrac{|f_1|}{1+|x|}\in L^1_x$ may be unbounded on every ball. The distinction is
consequential; see Remark~\ref{rem:whyD3star}.
\end{assumption}
\begin{remark}
\textup{(D2)}  requires the absolutely continuity of the
\emph{distributional} divergence:
\begin{equation}\label{eq:divac}
  \sum_i\partial_if_i \quad = (\Div f)\ \Leb ,
\end{equation}
$\Div f$ then denotes its density; $[\,\Div f\,]^-$ is the negative part. This is a standing hypothesis  in
Ambrosio's $\BV$ theory \cite{Amb2004,Ambrosio2008}.  Condition
\eqref{eq:divac} is \emph{not} implied by \textup{(D1)}, as a $\BV_\loc$ field may have a
divergence with a nonzero singular part, and it may not be weakened to a hypothesis on
the absolutely continuous part alone. 
   
\end{remark}
 We will use different combinations of these assumptions. We use the notations
 \begin{itemize}
     \item \textup{(D$_{\mathrm E}$)} for \textup{(D0)}, \textup{(D2)}, \textup{(D3)};
     \item \textup{(D$_{\mathrm E}^+$)} for  \textup{(D0)}, \textup{(D2)},
\textup{(D4)}; and
\item  \textup{(D$_{\mathrm L}$)} for \textup{(D0)}, \textup{(D1)}, \textup{(D2)}, \textup{(D3)}.
 \end{itemize}

\textup{(D$_{\mathrm E}$)} imposes \emph{no} derivative on $f$ beyond the scalar
distribution $\Div f$; in particular $f$ need not be weakly differentiable. This will be  enough for studying the Eulerian problem in Section \ref{sec.euler} (Theorem \ref{thm:eulerian}). By contrast, the Lagrangian result in Section \ref{sec.lagrange} needs the full matrix 
$\nabla f$ to be defined as a measure. 

\begin{assumption}[Initial distribution $\mu_0$\  (I)]\label{ass:I}\ \\
\begin{enumerate}
\item[(I1)] $\mu_0$ has a probability density  $p_0\in L^1\cap L^\infty(\R^d)$.
\item[(I2)]  $\int|x|^2p_0\,\dd x<\infty$ and $\Ent(p_0):=\int p_0\log p_0\,\dd x<\infty$.
\end{enumerate}
\end{assumption}

\begin{assumption}[Score regularity: (S)]\label{ass:S}
The marginal flow $(p_t, t\in [0,T])$ of \eqref{eq:sde} satisfies
\begin{enumerate}
\item[(S1)] $p_t>0$ $\Leb$-a.e.\ for each $t$, $p\in L^\infty\big((0,T);L^1\cap L^\infty\big)$,
$\sup_{0\leq t\le T}\int|x|^2p_t(x)dx <\infty$, and $t\mapsto p_t\Leb$ is narrowly continuous;
\item[(S2)] $\nabla\log p_\cdot\in L^1\big((0,T);W^{1,1}_\loc(\R^d;\R^d)\big)$;
\item[(S3)] $\big\|[\Delta\log p_t]^+\big\|_{L^\infty}\le\Lambda(t)$ with $\Lambda\in L^1(0,T)$;
\item[(S4)] $|\nabla\log p_s(x)|\le A(s)(1+|x|)$ for a.e.\ $x$, with $A\in L^1(0,T)$.
\end{enumerate}
We also record the weaker alternative to \textup{(S4)}:
\begin{enumerate}
\item[(S4$'$)] $\displaystyle\int_0^T\Fish(p_s)\,\dd s<\infty$, where
$\Fish(p):=\int_{\R^d}p\,|\nabla\log p|^2\,\dd x$ is the Fisher information.
\end{enumerate}
\end{assumption}

\section{Eulerian viewpoint: uniqueness of the density flow}\label{sec.euler}

Define the energy class
\[
  \mathcal{X}:=\Big\{p\in L^\infty\big([0,T];L^1\cap L^\infty(\R^d)\big)\ :\
  \nabla p\in L^2\big([0,T];L^2(\R^d)\big)\Big\}.
\]
A function $p\in\mathcal X$ is a \emph{weak solution} of the Fokker-Planck-Kolmogorov  equation \eqref{eq:fp} with initial condition  $p_0$ if
for all $\varphi\in C_c^\infty([0,T)\times\R^d)$
\[
  \int_0^T\!\!\!\int_{\mathbb{R}^d} p\,\partial_t\varphi+\int_{\mathbb{R}^d} p_0\varphi(0,\cdot)
  =-\int_0^T\!\!\!\int_{\mathbb{R}^d} p\,\langle f,\nabla\varphi\rangle
  +\tfrac12\int_0^Tdt\ \!\!\!\int_{\mathbb{R}^d}\sigma(t)^2\langle\nabla p_t,\nabla\varphi\rangle .
\]

\begin{theorem}[Existence and uniqueness of the marginal flow]\label{thm:eulerian} \ 
\begin{enumerate}
\item[\textup{(a)}] \emph{(Uniqueness)} Under assumptions  \textup{(E1)}, \textup{(D$_{\mathrm E}$)} and \textup{(I1)}, there is at most one weak solution of the Fokker-Planck-Kolmogorov equation \eqref{eq:fp}
in $\mathcal X$ with initial condition $p_0$.
\item[\textup{(b1)}] \emph{(Existence of solutions and $L^2$ estimate)}
Under assumptions  \textup{(E1)}, \textup{(D$_{\mathrm E}^+$)}, and \textup{(I1)} there exists a weak solution $p\in \mathcal X$ of \eqref{eq:fp}
 with initial condition $p_0$, and it  satisfies
\begin{align}
  \sup_{s\le T}\|p_s\|_{L^2}^2&\;\le\;\|p_0\|_{L^2}^2\,e^{\theta},
  \label{eq:energysup}\\[2pt]
  \epsilon^2\!\int_0^T\!\|\nabla p_s\|_{L^2}^2\,\dd s
&\;\le\;\|p_0\|_{L^2}^2\,\big(1+\theta\,e^{\theta}\big),
  \label{eq:energydiss}
\end{align}
where 
$\theta:=\int_0^T\|[\Div f(\cdot,s)]^-\|_{L^\infty}\,\dd s$.
In particular:
\begin{equation}\label{eq:energyest}
  \sup_{s\le T}\|p_s\|_{L^2}^2+\epsilon^2\!\int_0^T\!\|\nabla p_s\|_{L^2}^2\,\dd s
  \;\le\;2\,\|p_0\|_{L^2}^2\,e^{2\theta}.
\end{equation}
\item[\textup{(b2)}] \emph{(Moments and Fisher information)}
If, in addition,  \textup{(I2)} is satisfied then $$\sup_{t\leq T}\int_{\mathbb{R}^d}\|x\|^2 p_t(x) dx <\infty \qquad {\rm and}\qquad \int_0^T I(p_t)dt <\infty $$
\item[\textup{(c)}] \emph{(Identification)} If $X$ is any weak solution of the stochastic differential equation \eqref{eq:sde}
with $X_0\sim p_0\Leb$ whose   marginals are absolutely continuous with densities in
$\mathcal X$, then the density flow of $X$ is the (unique) weak solution of the Fokker-Planck-Kolmogorov equation \eqref{eq:fp} with initial condition $p_0$.
\end{enumerate}
\end{theorem}
No positivity, smoothness, log-concavity  is required for $p_0$ or $p_t$
and no  differentiability of $f$ beyond \textup{(D2)} is needed.

\begin{proof}
Uniqueness is a self-contained energy estimate; we give it in full, since it is the reason
the hypotheses can be as weak as they are.

\emph{Step 1: the difference of two solutions.} Let $p^1,p^2\in\mathcal X$ solve
\eqref{eq:fp} with the same initial condition  and put $w:=p^1-p^2\in\mathcal X$, so that $w(0)=0$ and $w$ is a weak solution of
\begin{equation}\label{eq:wdiff}
  \partial_sw+\Div(fw)-\tfrac12\sigma(s)^2\Delta w=0 .
\end{equation}
By interpolation $w\in L^\infty([0,T];L^2)$, and $\nabla w\in L^2([0,T];L^2)$.

\emph{Step 2: localised energy identity.} Fix $\phi\in C^\infty_c(\R^d)$ with
$0\le\phi\le1$, $\phi\equiv1$ on $B_1$, $\supp\phi\subset B_2$, and set
$\phi_R(x):=\phi(x/R)$. 
As
$w\in L^2_sH^1_x$ and, by \textup{(D0)} and $w\in L^\infty$, $fw\in L^2_\loc$, 
$\partial_tw\in L^2([0,T], H^{-1}(B_{2R}))$ and Lions--Magenes duality applies.
Testing \eqref{eq:wdiff} against $w\phi_R$ gives
\[
  \frac{\dd}{\dd s}\,\frac12\!\int\! w^2\phi_R
  +\frac{\sigma^2}{2}\!\int\!\phi_R|\nabla w|^2
  =\underbrace{\int \phi_R\,f\cdot w\nabla w}_{=:\mathrm{I}}
  +\underbrace{\int (f\cdot\nabla\phi_R)\,w^2}_{=:\mathrm{II}}
  -\underbrace{\frac{\sigma^2}{2}\!\int\! w\,\nabla w\cdot\nabla\phi_R}_{=:\mathrm{III}} .
\]

\emph{Step 3: the drift term  only depends on $\Div f$.}  Since
$w\in H^1_\loc\cap L^\infty$ we have $w\nabla w=\tfrac12\nabla(w^2)$ \emph{classically}, the
chain rule being available because membership of $\mathcal X$ already grants one weak
derivative --- no renormalisation is needed to justify it. Hence
$$\psi:=\varphi_Rw^2\in H^1_c(\mathbb R^d)\cap L^\infty(\mathbb R^d)\subset  W^{1,1}_c(\mathbb R^d).$$
The distributional-divergence identity, initially defined for
test functions in $C_c^\infty$, extends to this $\psi$ by
approximation. Indeed, one may choose
$\psi_n\in C_c^\infty(\mathbb R^d)$, supported in a fixed compact
set and satisfying
\[
\sup_n\|\psi_n\|_{L^\infty}\le \|\psi\|_{L^\infty},
\qquad
\psi_n\to\psi\ \text{in }H^1,
\qquad
\psi_n\to\psi\ \text{a.e.}
\]
Since $f\in L^2_{\mathrm{loc}}$, one has
\[
\int f\cdot\nabla\psi_n\to\int f\cdot\nabla\psi,
\]
while $\nabla\!\cdot f\in L^1_{\mathrm{loc}}$ and dominated
convergence give
\[
\int(\nabla\!\cdot f)\psi_n
\to
\int(\nabla\!\cdot f)\psi.
\]
Consequently,
\[
\int f\cdot\nabla\psi
=
-\int(\nabla\!\cdot f)\psi.
\]
Applying this identity with $\psi=\varphi_Rw^2$ gives
\[
I
=\frac12\int \varphi_R f\cdot\nabla(w^2)
=\frac12\int f\cdot\nabla(\varphi_Rw^2)
   -\frac12\int(f\cdot\nabla\varphi_R)w^2
=-\frac12\int(\nabla\!\cdot f)\varphi_Rw^2
   -\frac12\int(f\cdot\nabla\varphi_R)w^2.
\]
Note that only the scalar $\Div f$ is used, not the gradient $\nabla f$. It  is
paired with the bounded compactly supported function $\phi_Rw^2$. Two things are being used
here, and only the first is about $\nabla f$. First, the definition of the distributional
divergence gives $\int f\cdot\nabla\psi=-\langle\sum_i\partial_if_i,\psi\rangle$ with
$\psi=\phi_Rw^2$. Second --- and this is where \eqref{eq:divac} enters --- absolute
continuity of $\sum_i\partial_if_i$ lets us write that pairing as the Lebesgue integral
$\int(\Div f)\phi_Rw^2$. Without \eqref{eq:divac} the identity above acquires the extra term
$-\tfrac12\langle D^s\!f,\phi_Rw^2\rangle$, which has no sign, need not be absorbable, and
in general is not even well defined: $w^2$ is an $H^1_\loc\cap L^\infty$ function, specified
only up to $\Leb$-null sets, while $D^s\!f$ is carried by a $\Leb$-null set. Consequently
\[
  \mathrm{I}+\mathrm{II}
  \le\tfrac12\big\|[\Div f(\cdot,s)]^-\big\|_{L^\infty}\!\int w^2\phi_R
  \;+\;\tfrac12\Big|\int (f\cdot\nabla\phi_R)w^2\Big| .
\]

\emph{Step 4.} Write $f=f_1+f_2$ as in \textup{(D3)} with
$\tfrac{|f_1|}{1+|x|}\in L^1_sL^1_x$ and $\tfrac{|f_2|}{1+|x|}\in L^1_sL^\infty_x$. Since
$|\nabla\phi_R|\le\|\nabla\phi\|_\infty/R$ and $\nabla\phi_R$ is supported in
$\{R\le|x|\le2R\}$, where $1+|x|\le 1+2R\le 3R$,
\[
  \Big|\int (f\cdot\nabla\phi_R)w^2\Big|
  \;\lesssim\;\|w\|^2_{L^\infty}\!\!\int_{|x|\ge R}\!\frac{|f_1|}{1+|x|}
  \;+\;\Big\|\tfrac{|f_2|}{1+|x|}\Big\|_{L^\infty}\!\!\int_{|x|\ge R}\!w^2
  \;\xrightarrow[R\to\infty]{}\;0,
\]
both terms by dominated convergence. Likewise
$|\mathrm{III}|\le\tfrac{\epsilon^{-2}\|\nabla\phi\|_\infty}{2R}\|w\|_{L^2}\|\nabla w\|_{L^2}\to0$.

\emph{Step 5: Gr\"onwall.} Letting $R\to\infty$ and using $\sigma\ge\epsilon$,
\[
  \frac{\dd}{\dd s}\,\frac12\|w_s\|_{L^2}^2+\frac{\epsilon^2}{2}\|\nabla w_s\|_{L^2}^2
  \;\le\;\frac12\,\Theta_0(s)\,\|w_s\|_{L^2}^2 ,
\]
with $\Theta_0(s) =\|[\nabla.f(.,s)]^-\|_{L^\infty} \in L^1(0,T)$ by \textup{(D2)}. Since $w_0=0$, Gr\"onwall gives $w\equiv0$,
proving \textup{(a)}. Only the negative part of $\Div f$ was used, which is why
\textup{(D2)} is one-sided.

\emph{Step 6:    energy bounds.} Running Steps 2--5 with
$w$ replaced by $p$ itself yields the differential inequality
\[
  \frac{\dd}{\dd s}\,\frac12\|p_s\|_{L^2}^2+\frac{\epsilon^2}{2}\|\nabla p_s\|_{L^2}^2
  \;\le\;\frac12\,\Theta_0(s)\,\|p_s\|_{L^2}^2 .
\]
Discarding the (nonnegative) dissipation term and applying Gr\"onwall gives
\eqref{eq:energysup}. Retaining it instead and integrating over $[0,T]$,
\[
  \frac{\epsilon^2}{2}\!\int_0^T\!\|\nabla p_s\|_{L^2}^2
  \;\le\;\frac12\|p_0\|_{L^2}^2-\frac12\|p_T\|_{L^2}^2
  +\frac12\!\int_0^T\!\Theta_0(s)\|p_s\|_{L^2}^2\,\dd s
  \;\le\;\frac12\|p_0\|_{L^2}^2\big(1+\theta e^{\theta}\big),
\]
where \eqref{eq:energysup} was inserted in the last step; this is \eqref{eq:energydiss}.
Adding the two and using $1+(1+\theta)e^{\theta}\le 2e^{2\theta}$ gives
\eqref{eq:energyest}.

\emph{Step 7: existence.} This is not a standard construction: under
\textup{(D$_{\mathrm E}$)} the drift carries no pointwise bound and no weak derivative
beyond the scalar $\Div f$. We
need to mollify in $x$ and note that\\
\emph{(7a) The approximation preserves the hypotheses.} Let $\rho_\varepsilon$ be a standard
mollifier and $f^\varepsilon:=f*\rho_\varepsilon$ (in $x$, for a.e.\ $s$). Then
$f^\varepsilon\to f$ in $L^2_\loc$ by \textup{(D0)}. Because \eqref{eq:divac} holds, the
divergence is a function and commutes with convolution,
$\Div f^\varepsilon=(\Div f)*\rho_\varepsilon$, whence for a.e.\ $s$
\[
  \Div f^\varepsilon(x)=\int\Div f(y)\,\rho_\varepsilon(x-y)\,\dd y
  \;\ge\;-\big\|[\Div f(\cdot,s)]^-\big\|_{L^\infty},
  \qquad\text{so}\qquad
  \big\|[\Div f^\varepsilon]^-\big\|_{L^\infty}\le\big\|[\Div f]^-\big\|_{L^\infty}.
\]
The one-sided bound is therefore preserved with the same constant, uniformly in
$\varepsilon$. Finally, since $1+|x|\le2(1+|y|)$ whenever
$|x-y|\le\varepsilon\le1$, \textup{(D3)} passes to $f^\varepsilon$ with its constants at most
doubled, and \textup{(D4)} passes to $f^\varepsilon$ with $C_f$ replaced by $2C_f$,
since $|f^\varepsilon(x)|\le(|f|*\rho_\varepsilon)(x)\le C_f(s)\big(1+|x|+\varepsilon\big)$.
Under condition (I1), there exists $p_0^\varepsilon\in C_c^\infty(\mathbb R^d)$,
$p_0^\varepsilon\ge0$, such that
\[
\int_{\mathbb R^d}p_0^\varepsilon(x)\,dx=1,\quad p_0^\varepsilon\to p_0 \quad\text{in }L^1(\mathbb R^d),
\qquad
\|p_0^\varepsilon\|_{L^\infty}
\le \|p_0\|_{L^\infty}.
\]
\emph{This is where \textup{(D4)} is needed.}
Under \textup{(D4)} the field $f^\varepsilon$ is smooth with globally linear growth,
so the regularised Cauchy problem on $\R^d$ --- uniformly parabolic, constant diffusion,
smooth drift of linear growth --- has a unique  solution $p^\varepsilon$, mass is
conserved, and the Phragm\'en--Lindel\"of form of the maximum principle applies, which is
what (7b) below uses.  See Remark~\ref{rem:whyD3star}.

To show (b2) we note that, under $(I2)$, the approximant $p^\varepsilon$ may be chosen to satisfy 
$$ \sup_{\varepsilon>0}
\int_{\mathbb R^d}|x|^2p_0^\varepsilon(x)\,dx<\infty $$
and step (7c') gives uniform bounds on the second moment.\\
\emph{(7b) Positivity and $L^\infty$ control.} $p^\varepsilon\geq 0$ by the
parabolic maximum principle. Write the equation in nondivergence form:
$$\partial_sp^\varepsilon+f^\varepsilon\!\cdot\!\nabla p^\varepsilon+(\Div f^\varepsilon)p^\varepsilon
=\tfrac{\sigma^2}{2}\Delta p^\varepsilon.$$ At a spatial maximum $\nabla p^\varepsilon=0$ and
$\Delta p^\varepsilon\le0$, so $\partial_sp^\varepsilon\le-(\Div f^\varepsilon)p^\varepsilon
\le\|[\Div f^\varepsilon]^-\|_{L^\infty}p^\varepsilon$, and Gr\"onwall with (7a) gives
\begin{equation}\label{eq:linfty}
  \big\|p^\varepsilon_s\big\|_{L^\infty}\;\le\;\|p_0\|_{L^\infty}\,e^{\theta},
  \qquad\text{uniformly in }\varepsilon .
\end{equation}
Only the one-sided bound is used, consistently with \textup{(D2)}.

\emph{(7c) Mass conservation and tightness.} Mass is conserved for the regularised problem.
For tightness the linear cut-off $\phi_R$ of Step 4 is \emph{not} adequate: the drift term
produces the annular mass $\int_{\{R\le|x|\le2R\}}p^\varepsilon$, which is not controlled by
$\int p^\varepsilon(1-\phi_R)$, since $1-\phi_R$ vanishes at the inner edge $|x|=R$ where that
mass may be concentrated. We therefore use a \emph{logarithmic} cut-off, for which the
drift term is small outright and no Gr\"onwall argument is used.
Let $\chi\in C^\infty\big([0,\infty);[0,1]\big)$ with $\chi\equiv1$ on $[0,\tfrac12]$ and
$\chi\equiv0$ on $[1,\infty)$, and for $R>e$ set
\begin{equation}\label{eq:logcut}
  \psi_R(x):=\chi\Big(\frac{\log(1+|x|)}{\log R}\Big),
\end{equation}
so that $\psi_R\equiv1$ on $\{1+|x|\le\sqrt R\}$ and $\psi_R\equiv0$ on $\{1+|x|\ge R\}$.
Writing $\hat x=x/|x|$, one has
$\nabla\psi_R=\tfrac{\chi'}{\log R}\tfrac{\hat x}{1+|x|}$ and
$\Delta\psi_R=\tfrac{\chi''}{(\log R)^2(1+|x|)^2}+\tfrac{\chi'}{\log R}\Delta\log(1+|x|)$
with $\Delta\log(1+|x|)=\tfrac{d-1}{|x|(1+|x|)}-\tfrac{1}{(1+|x|)^2}$; since
$\nabla\psi_R$ is supported in $\{|x|\ge\sqrt R-1\}$, this gives for $R$ large
\begin{equation}\label{eq:logcutbound}
  \big(1+|x|\big)\,\big|\nabla\psi_R(x)\big|\;\le\;\frac{\|\chi'\|_\infty}{\log R},
  \qquad
  \big(1+|x|\big)^2\big|\Delta\psi_R(x)\big|
  \;\le\;\frac{C_d\big(\|\chi'\|_\infty+\|\chi''\|_\infty\big)}{\log R}.
\end{equation}
$(1+|x|)|\nabla\psi_R|\to0$ uniformly, whereas
for the linear cut-off $(1+|x|)|\nabla\phi_R|$ is merely bounded.

Now test the regularised equation against $\psi_R$. Splitting $f^\varepsilon=f_1^\varepsilon+f_2^\varepsilon$
with $f_i^\varepsilon:=f_i*\rho_\varepsilon$, which by (7a) satisfies
$\big\|\tfrac{|f_2^\varepsilon(\cdot,s)|}{1+|x|}\big\|_{L^\infty}\le2\kappa_2(s)$ and
$\big\|\tfrac{|f_1^\varepsilon(\cdot,s)|}{1+|x|}\big\|_{L^1}\le2\kappa_1(s)$ with
$\kappa_1,\kappa_2\in L^1(0,T)$ furnished by \textup{(D3)},
\[
  \Big|\frac{\dd}{\dd s}\!\int p^\varepsilon_s\psi_R\Big|
  \;\le\;\underbrace{\int p^\varepsilon_s\,\big|f_2^\varepsilon\big|\,|\nabla\psi_R|}_{\le\,2\kappa_2(s)\|\chi'\|_\infty/\log R}
  \;+\;\underbrace{\int p^\varepsilon_s\,\big|f_1^\varepsilon\big|\,|\nabla\psi_R|}_{\le\,2\kappa_1(s)\|p^\varepsilon\|_{L^\infty}\|\chi'\|_\infty/\log R}
  \;+\;\underbrace{\tfrac{\sigma^2}{2}\!\int p^\varepsilon_s|\Delta\psi_R|}_{\le\,C_{d,\epsilon}/\log R},
\]
where the first and third bounds use $\int p^\varepsilon_s=1$ together with
\eqref{eq:logcutbound}, and the second uses $\|p^\varepsilon\|_{L^\infty}$ from
\eqref{eq:linfty}. Every term carries the factor $(\log R)^{-1}$ and \emph{none} involves the
exterior mass, so integrating in $s$,
\[
  \sup_{s\le T}\Big|\int p^\varepsilon_s\psi_R-\int p^\varepsilon_0\psi_R\Big|
  \;\le\;\frac{C}{\log R},
\]
with $C$ depending only on $d,\epsilon,\chi$, $\|\kappa_i\|_{L^1(0,T)}$, $\|p_0\|_{L^\infty}$
and $\theta$ --- in particular not on $\varepsilon$ or $R$. Since
$\int p_0^\varepsilon\psi_R\to1$ as $R\to\infty$ uniformly in $\varepsilon$ (as
$p_0^\varepsilon\to p_0$ in $L^1$ and $\psi_R\uparrow1$), and $\psi_R$ vanishes off
$\{1+|x|<R\}$,
\[
  \sup_\varepsilon\ \sup_{s\le T}\int_{\{1+|x|\ge R\}}p^\varepsilon_s\,\dd x
  \;\le\;1-\inf_\varepsilon\int p_0^\varepsilon\psi_R+\frac{C}{\log R}
  \;\xrightarrow[R\to\infty]{}\;0 .
\]
The family is therefore uniformly tight, no mass escapes in the limit, and
$\|p_s\|_{L^1}=1$. Note that the $L^1$ component of \textup{(D3)} is again handled with no
pointwise bound on $f$: it is paired with $(1+|x|)|\nabla\psi_R|$, which
\eqref{eq:logcutbound} makes uniformly small.\\
\medskip
\noindent

\emph{(7c$'$) Uniform second moments.}
By It\^o's formula for the regularised SDE, or equivalently by testing
against bounded $C^2$ truncations of $|x|^2$ and then applying
monotone convergence,
\[
\frac{d}{ds}
\int_{\mathbb R^d}|x|^2p_s^\varepsilon(x)\,dx
=
2\int_{\mathbb R^d}
x\cdot f^\varepsilon(x,s)p_s^\varepsilon(x)\,dx
+d\,\sigma(s)^2.
\]
By $(D4)$ and (7a), $
|f^\varepsilon(x,s)|
\le 2C_f(s)(1+|x|).$
Since $|x|(1+|x|)\le 1+2|x|^2$, writing
\[
m_2^\varepsilon(s)
:=
\int_{\mathbb R^d}|x|^2p_s^\varepsilon(x)\,dx,
\]
we obtain
\[
\frac{d}{ds}m_2^\varepsilon(s)
\le
C_d\bigl(1+C_f(s)\bigr)
\bigl(1+m_2^\varepsilon(s)\bigr).
\]
Hence, by Gr\"onwall's inequality,
\[
M:=
\sup_{\varepsilon>0}\sup_{s\in[0,T]}
m_2^\varepsilon(s)<\infty.
\]
By lower semicontinuity,
\[
\sup_{s\in[0,T]}
\int_{\mathbb R^d}|x|^2p_s(x)\,dx
\le M.
\]
\emph{(7d) Uniform energy bounds.} Steps 2--5 applied to $p^\varepsilon$ are legitimate
classically and, by (7a), yield \eqref{eq:energysup}--\eqref{eq:energydiss} with constants
independent of $\varepsilon$. In particular $\nabla p^\varepsilon$ is bounded in $L^2L^2$.

\emph{(7e) Compactness of the flux $f^\varepsilon p^\varepsilon$.} This is the only step
where a product of two weakly convergent sequences must be passed to the limit, and it needs
strong compactness of $p^\varepsilon$. For $\psi\in C_c^\infty(K)$,
\[
  \big|\langle\partial_sp^\varepsilon,\psi\rangle\big|
  \le\big\|f^\varepsilon p^\varepsilon\big\|_{L^2(K)}\|\nabla\psi\|_{L^2}
  +\tfrac{\epsilon^{-2}}{2}\|\nabla p^\varepsilon\|_{L^2}\|\nabla\psi\|_{L^2},
  \qquad
  \big\|f^\varepsilon p^\varepsilon\big\|_{L^2(K)}\le\|f^\varepsilon\|_{L^2(K)}\|p^\varepsilon\|_{L^\infty},
\]
which is bounded by \textup{(D0)} and \eqref{eq:linfty}. Hence $\partial_sp^\varepsilon$ is
bounded in $L^2\big([0,T];H^{-1}(K)\big)$ while $p^\varepsilon$ is bounded in
$L^2\big([0,T];H^1(K)\big)$, and Aubin--Lions gives, along a subsequence,
$p^\varepsilon\to p$ strongly in $L^2\big([0,T];L^2(K)\big)$ for every compact $K$.
Combined with $f^\varepsilon\to f$ strongly in $L^2_\loc$, this gives
$f^\varepsilon p^\varepsilon\to fp$ in $L^1_\loc$, and every term of the weak formulation
passes to the limit. \textup{(D0)} is used here.

\emph{(7f) Initial trace.} The bound on $\partial_sp^\varepsilon$ in $L^2(H^{-1}(K))$ makes
$\{p^\varepsilon\}$ equicontinuous in $C\big([0,T];H^{-1}(K)\big)$; hence
$p\in C\big([0,T];H^{-1}_\loc\big)$ and $p_s\to p_0$ in $H^{-1}_\loc$ as $s\downarrow0$, so
the initial condition is attained. Together with
\eqref{eq:linfty}, (7c) and (7d), the limit $p$ lies in $\mathcal X$ and solves \eqref{eq:fp}.

\emph{Step 8: identification.} For \textup{(c)}, let
$X$ be a weak solution of \eqref{eq:sde} with marginal densities $q_s\in\mathcal X$. It\^o's
formula applied to $\varphi\in C^\infty_c$ and integration in $x$ show that $(q_s)$ is a
weak solution of \eqref{eq:fp} with initial condition  $p_0$.
By part~(a), $(q_s)$ is the unique weak solution of \eqref{eq:fp}
with initial condition $p_0$. If $(D_E^+)$ also holds, it coincides
with the solution furnished by part~(b1).
\end{proof}

\begin{remark}[Existence needs \textup{(D4)} while uniqueness does not]
\label{rem:whyD3star}
Uniqueness in Theorem~\ref{thm:eulerian}\textup{(a)} sees the drift only through the two
integrals $\int\phi_R\,f\cdot\nabla(w^2)$ and $\int(f\cdot\nabla\phi_R)w^2$ of Steps 3--4,
in both of which $f$ is paired with a bounded, compactly supported weight. The decomposition
\textup{(D3)} is tailored to this: the $L^1_x$ component is integrated, never
evaluated, and the factor $3R$ from the annulus cancels the $R^{-1}$ from $\nabla\phi_R$.

Existence is different because it requires the \emph{approximating} problems to be
solvable, mass-conserving, and to obey a maximum principle on  $\R^d$, which are
pointwise statements about $f^\varepsilon$. under (D3) alone, $f^{\varepsilon}$
 carries no pointwise bound and  truncation  would  destroy the divergence estimate. \textup{(D4)} provides such a pointwise bound.
 \end{remark}

\begin{remark}[Absolute continuity of the divergence is necessary]
\label{rem:divsing}
If \eqref{eq:divac}  
is dropped and (D2) is reinterpreted as a hypothesis on the absolutely continuous part alone
then the conclusion of Theorem~\ref{thm:eulerian}\textup{(b)}
does not hold. Take $d=1$ and
\[
  \sigma\equiv1,\qquad f(x)=-\operatorname{sign}(x),\qquad p_0(x)=e^{-2|x|} .
\]
Then $\int p_0=1$, and $f$ satisfies \textup{(D0)} ($f\in L^\infty$), \textup{(D1)}
($\operatorname{sign}\in\BV_\loc$) and \textup{(D3)} ($|f|\le1$). Its distributional
divergence is $\partial_xf=-2\delta_0$, purely singular, so its absolutely continuous part is zero
$\Div f\equiv0$, giving $\Theta_0\equiv0$ and $\textup{(D2)}$.
The flux vanishes identically, since
$\partial_xp_0=-2\operatorname{sign}(x)e^{-2|x|}$ gives
\[
  f\,p_0-\tfrac12\partial_xp_0=-\operatorname{sign}(x)e^{-2|x|}+\operatorname{sign}(x)e^{-2|x|}=0 ,
\]
so $p_s\equiv p_0$ solves \eqref{eq:fp} and lies in $\mathcal X$. Here
$\|p_0\|_{L^2}^2=\tfrac12$, $\|\partial_xp_0\|_{L^2}^2=2$, $\epsilon=1$ and $\theta=0$, so
the dissipation bound \eqref{eq:energydiss} would assert
\[
2T\;=\;\epsilon^2\!\int_0^T\!\|\partial_xp_s\|_{L^2}^2\,\dd s
  \;\le\;\|p_0\|_{L^2}^2\big(1+\theta e^{\theta}\big)=\tfrac12 .
\]
The left side grows linearly in $T$ while the right side does not depend on $T$ at all, so
this is false for every $T>\tfrac14$, with an unbounded deficit as $T\to\infty$.  The discrepancy is exactly the term dropped in Step 3: with
$D^s\!f=-2\delta_0$,
\[
  -\tfrac12\big\langle D^s\!f,\,p^2\big\rangle=p_0(0)^2=1
  =\int_\R f\,p_0\,\partial_xp_0\,\dd x=\tfrac12\|\partial_xp_0\|_{L^2}^2 ,
\]
so restoring it turns the false inequality into the exact stationary energy identity. Note
that the singular divergence here is \emph{negative}, i.e.\ compressing; the drift
$-\operatorname{sign}(x)$ funnels mass into the origin at a rate the a.c.\ part cannot see.
\end{remark}

\begin{remark}\label{rem:buys}
Theorem~\ref{thm:eulerian} holds for $p_0$  bounded and integrable. No smoothness  or other property is required. In particular
data may be concentrated near a lower-dimensional set: the density flow is
still uniquely determined. As we will see in the next section, the same is not true  for the Lagrangian flow.
\end{remark}

\section{ Lagrangian viewpoint: the probability flow ODE}\label{sec.lagrange}
We now turn to the probability flow ODE \eqref{eq:pfode}.

We recall the concept of regular Lagrangian flow, following DiPerna and Lions \cite{DL1989}:
\begin{definition}[Regular Lagrangian flow]\label{def:rlf}
$Z:[0,T]\times\R^d\to\R^d$ is a \emph{regular Lagrangian flow} (RLF) for a velocity field $v$ if
\begin{enumerate}
    \item[(i)] for $\Leb$-a.e.\ $x$, $t\mapsto Z(t,x)$ is absolutely continuous with
$Z(t,x)=x+\int_0^tv(Z(s,x),s)\dd s$;
\item[(ii)] $Z(t,\cdot)_\#\Leb\le L\,\Leb$ for some $L\ge1$.
\end{enumerate} 
\end{definition}

\begin{definition}[Bounded distributional solution]\label{def:class}
For $u_0\in L^1\cap L^\infty(\R^d)$, a \emph{bounded distributional solution} of
\eqref{eq:cont} with initial condition  $u_0$ is a function
$u\in L^\infty\big((0,T);L^1\cap L^\infty(\R^d)\big)$ with $uv\in L^1_\loc$ such that
\[
  \int_0^T\!\!\int_{\R^d}\big(\partial_s\varphi+v\cdot\nabla\varphi\big)\,u\,\dd x\,\dd s
  \;+\;\int_{\R^d}\varphi(\cdot,0)\,u_0\,\dd x\;=\;0
  \qquad\forall\,\varphi\in C_c^\infty\big(\R^d\times[0,T)\big).
\]
This formulation presupposes $v\in L^1_\loc([0,T]\times\R^d)$, without which the second
term in the integrand is undefined for general $u\in L^\infty$; see
Theorem~\ref{thm:ce-fail}\textup{(i)}.
\end{definition}

The following result summarizes Ambrosio's BV renormalisation theorem,
comparison principle, and existence and uniqueness theory for regular
Lagrangian flows; see
\cite[Theorems~3.6, 4.3, 5.1 and~6.1]{Ambrosio2008},
which extend previous results by DiPerna and Lions
\cite[Theorem~III.2]{DL1989} derived for
$\nabla\!\cdot v\in L^1((0,T);L^\infty)$.
\begin{theorem}[Ambrosio/DiPerna-Lions]
\label{thm:DL}
Let $v:\mathbb R^d\times(0,T)\to\mathbb R^d$ be a vector
field satisfying
\begin{align*}
\textnormal{(R1)}\quad&
v\in L^1\bigl((0,T);BV_{\mathrm{loc}}
(\mathbb R^d;\mathbb R^d)\bigr),\\
\textnormal{(R2)}\quad&
D\!\cdot v_t=(\nabla\!\cdot v_t)\mathcal L^d
\quad\text{for a.e. }s,
\qquad
[\nabla\!\cdot v]^-
\in L^1\bigl((0,T);L^\infty(\mathbb R^d)\bigr),\\
\textnormal{(R3)}\quad&
\frac{|v|}{1+|x|}
\in
L^1\bigl((0,T);L^1(\mathbb R^d)\bigr)
+
L^1\bigl((0,T);L^\infty(\mathbb R^d)\bigr).
\end{align*}
Then:
\begin{enumerate}
\item
For every $u_0\in L^1\cap L^\infty(\mathbb R^d)$, the continuity
equation
\[
\partial_tu+\nabla\!\cdot(vu)=0,
\qquad
u(.,0)=u_0,
\]
has at most one bounded distributional solution 
$
u\in L^\infty\bigl((0,T);L^1\cap L^\infty(\mathbb R^d)\bigr).
$
\item
There exists a regular Lagrangian flow $Z$ associated with $v$,
unique up to $\mathcal L^d$-null sets, and
\[
Z(t,\cdot)_\#\mathcal L^d
\le
\exp\left(
\int_0^t
\bigl\|[\nabla\!\cdot v(\cdot,s)]^-\bigr\|_{L^\infty}\,ds
\right)\mathcal L^d.
\]
\item
For every nonnegative
$u_0\in L^1\cap L^\infty(\mathbb R^d)$, the unique solution is
$
u_t\mathcal L^d
= Z(t,\cdot)_\#(u_0\mathcal L^d).$
\end{enumerate}
If, in addition,
\[
[\nabla\!\cdot v]^+
\in L^1\bigl((0,T);L^\infty(\mathbb R^d)\bigr),
\]
then the reverse flow is well posed and satisfies two-sided compression bounds;
in particular $Z(t,\cdot)$ is
essentially invertible.
\end{theorem}
The uniqueness assertion follows from
\cite[Theorems~4.3 and~5.1]{Ambrosio2008}:
Theorem~5.1 yields renormalisation, and Theorem~4.3 yields the
comparison principle; applying comparison in both directions gives
uniqueness.

\begin{theorem}[Superposition \cite{Amb2004,AGS,AT2014}]\label{thm:super}
If $t\mapsto \mu_t$ is narrowly continuous and satisfies $$\partial_t\mu_t+\Div(v\mu_t)=0,\quad {\rm and}\qquad 
\int_0^T\!\int\frac{|v|}{1+|x|}\ \dd\mu_t\dd t<\infty, $$ then there is
$\eta\in\mathcal P(C([0,T];\R^d))$ concentrated on integral curves of $v$ with $(e_t)_\#\eta=\mu_t$, where $e_t(\gamma)=\gamma(t)$ is the evaluation map.
\end{theorem}
The following result uses the Ambrosio-Diperna-Lions theory  to study the PF-ODE \eqref{eq:pfode} and the associated continuity equation \eqref{eq:cont}.
Theorem~\ref{thm:eulerian} did not require any regularity assumption on  the score function $\nabla\log p_t$; here we
cannot avoid such assumptions. The  PF-ODE \eqref{eq:pfode} is associated with the  velocity field  $v=f-\tfrac12\sigma^2\nabla\log p$ so one does require regularity of $\nabla\log p_t$.
\begin{theorem}[Lagrangian well-posedness and transport for the PF-ODE]\label{thm:lagrangian}
Assume \textup{(E1)}, \textup{(D$_{\mathrm L}$)}, \textup{(I)} and \textup{(S)}, and let
$v(x,t)=f(x,t)-\tfrac12\sigma^2(t)\nabla\log p_t(x)$. Then:
\begin{enumerate}
\item[\textup{(i)}] $v$ satisfies conditions \textup{(R1)--(R3)} of Theorem \ref{thm:DL}, with
\[
  \big\|[\Div v(\cdot,t)]^-\big\|_\infty\le\big\|[\Div f(\cdot,t)]^-\big\|_\infty
  +\tfrac12\|\sigma\|^2_{L^\infty}\Lambda(t)=:\Theta_-(t)\in L^1(0,T).
\]
\item[\textup{(ii)}] The PF-ODE \eqref{eq:pfode} admits a regular Lagrangian flow $Z$, unique up to $\Leb$-null
sets, with compressibility constant $L=e^{\Theta_-}$, where 
$\Theta_-:=\int_0^T\Theta_-(t)\,\dd t$; 
\item[\textup{(iii)}] $Z(t,\cdot)_\#(p_0\Leb)=p_t\Leb$ for every $t\in[0,T]$; hence for
$Z_0\sim p_0$ and $Z_t:=Z(t,Z_0)$,
\[
  \Law(Z_t)=\Law(X_t),\qquad t\in[0,T].
\] 
\end{enumerate}
\end{theorem}
The asymmetry between 
Theorems~\ref{thm:eulerian} and~\ref{thm:lagrangian} is due to the fact that the diffusion model benefits from 'regularisation by noise' \cite{Ver1981,KR2005,FGP2010}  which leads to well-posedness of the SDE even in absence of regularity of the velocity field.
The PF-ODE \eqref{eq:pfode} does not benefit from this effect, so well-posedness of its associated probability flow requires many more assumptions.

\begin{remark}\label{rem:overkill} As we will see in Section~\ref{sec:verify}, under certain additional assumptions ensuring that the score is globally
Lipschitz and $f$ is smooth,  the probability-flow velocity $v$ is  Lipschitz, so classical Cauchy–Lipschitz theory already yields a unique flow. The significance of Theorem~\ref{thm:lagrangian}  is  not to recover this classical setting, but to extend well-posedness to Sobolev and bounded-variation score functions, where classical ODE theory no longer applies. This extension is essential for treating diffusion models beyond the uniformly Lipschitz regime.  $f$ need only be locally of bounded variation with one-sided divergence
bound, while $(S3)$ requires control only of
$[\Delta\log p_s]^+$, rather than of the full log-Hessian. The price to pay is that $Z$ is defined only $\Leb$-a.e.
\end{remark}

\begin{remark}[The PF-ODE only replicates the marginal flow]
Theorem~\ref{thm:lagrangian}(iii) states that the PF-ODE replicates the  flow of marginal distributions of $X$.
Conditionally on
$Z_0$, the law of $(Z_s)_{s\le T}$ is a concentrated  on a single curve, so is  singular with
respect to the law of $X=(X_t)_{t\in[0, T]}$; all conditional distributions for $X$ and $Z$ are in fact mutually singular! What is preserved
is exactly $\{\Law(X_t)\}_{t\in[0,T]}$, which is what a sampler requires.
\end{remark}
We now turn to the proof of Theorem~\ref{thm:lagrangian}.
\begin{proof}
By It\^o's formula and integration by parts,
$p$ is a weak solution of \eqref{eq:fp}. By (S1)--(S2), $p_s>0$ a.e.\ and
$\nabla\log p_s\in W^{1,1}_\loc$, so $\nabla p_s=p_s\nabla\log p_s$ in $L^1_\loc$ and
$\tfrac12\sigma^2\Delta p_s=\tfrac12\sigma^2\Div(p_s\nabla\log p_s)$ in $\mathcal D'$.
Substituting gives $\partial_sp+\Div(pv)=0$, with $pv\in L^1_\loc$ by (S1), (S4) and the
uniform second-moment bound.

(R1): (D1) gives $f\in L^1(\BV_\loc)$ and (S2) gives
$\tfrac12\sigma^2\nabla\log p\in L^1(W^{1,1}_\loc)\subset L^1(\BV_\loc)$, using
$\sigma^2\le\epsilon^{-2}$. (R2): $\Div v=\Div f-\tfrac12\sigma^2\Delta\log p$, so
$[\Div v]^-\le[\Div f]^-+\tfrac12\sigma^2[\Delta\log p]^+$, and (D2), (E1), (S3) give the
stated $\Theta_-$. \\
For \emph{(R3)}: by (D3), $f=f_1+f_2$ with
\[  
\frac{f_1}{1+|x|}
\in L^1\bigl((0,T);L^1(\mathbb R^d)\bigr),
\qquad
\frac{f_2}{1+|x|}
\in L^1\bigl((0,T);L^\infty(\mathbb R^d)\bigr).
\] 
Let $v_1=f_1,$ and $
v_2=f_2-\frac12\sigma^2\nabla\log p.$
Then
\[
\frac{|v_1|}{1+|x|}
=
\frac{|f_1|}{1+|x|}
\in L^1\bigl((0,T);L^1(\mathbb R^d)\bigr),
\]
while \emph{(S4)} and $\sigma^2\le\epsilon^{-2}$ give
\[
\frac{|v_2(x,s)|}{1+|x|}
\le
\left\|
\frac{f_2(\cdot,s)}{1+|\cdot|}
\right\|_{L^\infty}
+\frac12\epsilon^{-2}A(s)
=G(s).
\]
Thus \emph{(R3)} holds.

\emph{(ii).} We apply Theorem~\ref{thm:DL}. If in addition
$[\Div v]^+\in L^1(L^\infty)$ one obtains the two-sided Jacobian bound
\begin{equation}\label{eq:jac}
  e^{-\Theta_+}\,\Leb\;\le\;Z(s,\cdot)_\#\Leb\;\le\;e^{\Theta_-}\,\Leb,
  \qquad
  \Theta_\pm:=\int_0^T\big\|[\Div v(\cdot,s)]^{\pm}\big\|_{L^\infty}\,\dd s,
\end{equation}
and $Z(s,\cdot)$ is essentially injective. In the smooth case the Jacobian
$J(s,x)=\det\nabla_xZ(s,x)$ solves $\partial_sJ=(\Div v)(Z,s)J$, so
$e^{-\Theta_-}\le J\le e^{\Theta_+}$; since the pushforward density at $Z(s,x)$ is $1/J(s,x)$,
the \emph{upper} bound on $Z(s,\cdot)_\#\Leb$ is governed by $\Theta_-$ and the \emph{lower}
bound by $\Theta_+$, as displayed. 
Under \textup{(D2)} and \textup{(S3)}, which are one-sided bounds, only
$\Theta_-<\infty$ is obtained. The complementary estimate
$\Theta_+<\infty$ does not follow from these assumptions; it requires,
for example, control of $[\nabla\!\cdot f]^+$ and
$[\Delta\log p]^-$. Under the additional hypothesis
$[\nabla\cdot{}v]^+\in L^1(L^\infty)$ the lower compression bound holds and the reverse
flow is well posed; in particular, $Z(s,\cdot)$ is essentially
invertible.

\emph{(iii).} Let $
\mu_s:=Z(s,\cdot)_\#(p_0\mathcal L^d).$
By Theorem~\ref{thm:DL}, $(\mu_s)$ is a weak solution of
\[
\partial_s\mu_s+\nabla\!\cdot(v\mu_s)=0,
\qquad
\mu_0=p_0\mathcal L^d.
\]
Moreover, by the compressibility estimate,
\[
\mu_s
\le
\|p_0\|_{L^\infty}
\,Z(s,\cdot)_\#\mathcal L^d
\le
\|p_0\|_{L^\infty}
\exp\left(
\int_0^s
\bigl\|[\nabla\!\cdot v(\cdot,r)]^-\bigr\|_{L^\infty}\,dr
\right)\mathcal L^d.
\]
Thus $\mu_s=u_s\mathcal L^d$, where
\[
u\in
L^\infty\bigl((0,T);L^1\cap L^\infty(\mathbb R^d)\bigr),
\qquad
\|u_s\|_{L^1}=1,
\qquad
\|u_s\|_{L^\infty}
\le e^{\Theta^-}\|p_0\|_{L^\infty}.
\]
Hence $u$ belongs to the uniqueness class of Theorem~\ref{thm:DL}.
By the first paragraph of the proof, $p$ is also a bounded
weak solution of the same continuity equation with initial
condition $p_0$. Uniqueness therefore gives
$
u=p
\quad\text{a.e. on }(0,T)\times\mathbb R^d.$
The curves
\[
s\longmapsto \mu_s,
\qquad
s\longmapsto p_s\mathcal L^d
\]
are narrowly continuous; therefore their equality for a.e. $s$
extends to every $s\in[0,T]$. Thus
\[
Z(s,\cdot)_\#(p_0\mathcal L^d)=p_s\mathcal L^d
\qquad\text{for every }s\in[0,T].
\]
If $Z_0\sim p_0\mathcal L^d$ and $Z_s:=Z(s,Z_0)$, then
$
\operatorname{Law}(Z_s)
=
p_s\mathcal L^d
=
\operatorname{Law}(X_s).$
\end{proof}

\begin{remark}[Replacing \textup{(S4)} by \textup{(S4$'$)}]\label{rem.iv}
    If assumption \textup{(S4)} is replaced by the weaker assumption  \textup{(S4$'$)},     then
conclusions \textup{(ii)} and \textup{(iii)} may fail but  $(p_t)$ admits a
probabilistic representation as a superposition of integral curves of $v$
\textup{(}Theorem~\ref{thm:super}\textup{)}.
    Indeed, an application of the Cauchy--Schwarz inequality yields
\[
  \int_0^T\!\!\!\int\frac{|v|}{1+|x|}\,p_s\,\dd x\,\dd s
  \le\int_0^T\!\!\!\int\frac{|f|}{1+|x|}p_s
+\tfrac12\epsilon^{-2}\!\int_0^T\!\!\!\int|\nabla\log p_s|\,p_s
  \le C+\tfrac12\epsilon^{-2}\sqrt T\Big(\int_0^T\!\Fish(p_s)\dd s\Big)^{1/2}<\infty,
\]
the first term being finite by (D3) and (S1). So the integrability assumption of
Theorem~\ref{thm:super} holds and it yields the representation. It does \emph{not} yield
(R3), which is a pointwise statement about $v$ and not implied by a
$p_t$-weighted bound; hence neither RLF nor uniqueness of $\eta$ can be asserted.
\end{remark}
The following lemma shows that finite initial entropy and second-moment   plus the
one-sided divergence bound imply (S4$'$):\begin{lemma}[Entropy dissipation and $(S4')$]
Assume \emph{(E1)}, \emph{$(D_E^+)$} and \emph{(I1)--(I2)}, and let
$p$ be the solution constructed in Theorem~\ref{thm:eulerian}(b), i.e.\ the limit of
the regularized scheme of Step~7. Then
\begin{equation}
\frac{\epsilon^2}{2}\int_0^T I(p_s)\,ds
\le
\log\!\bigl(1+\|p_0\|_{L^\infty}\bigr)
+1
+C_d(1+M)
+
\bigl\|[\nabla\!\cdot f]^-\bigr\|_
{L^1([0,T];L^\infty)}
<\infty,
\label{eq.14}
\end{equation}
where
\[
M:=
\sup_{\varepsilon>0}\sup_{s\in[0,T]}
\int_{\mathbb R^d}|x|^2p_s^\varepsilon(x)\,dx
<\infty
\]
is the uniform second-moment bound obtained in Step~$7(c')$.
In particular, $(S4')$ holds.\label{lem:fisher}
\end{lemma}
The proof of this lemma is given in Appendix \ref{app.fisher}.

The above results apply to score-based diffusion models \cite{Song2021,hanxu2024}:
\begin{corollary}[Score-based generative models]\label{cor:diffusion}
Consider the variance-preserving SDE  
\begin{equation}\label{eq:vpbeta}
f(x,s)=-\beta(s)x,\quad  \sigma=\sqrt{2\beta},\qquad 0<\beta_{\min}\le\beta(s)\le\beta_{\max}<\infty\qquad\text{for a.e.\ }s\in[0,T],
\end{equation}
or the variance-exploding SDE ($f\equiv0$) with $\sigma$ satisfying \textup{(E1)}. Let
$p_0\in L^1\cap L^\infty(\R^d)$ satisfy (I2) i.e. have finite entropy and finite second moment. Then:
\begin{enumerate}
\item[\textup{(i)}] \textup{(E1)}, \textup{(D$_{\mathrm L}$)} and \textup{(I)} hold. For the
variance-preserving model \eqref{eq:vpbeta} is exactly what \textup{(E1)} requires, since
$\sigma^2=2\beta$; the upper bound alone does not suffice.
\item[\textup{(ii)}] If, in addition, $p_0$ satisfies the two-sided log-Hessian condition
\eqref{eq:p0hess} of Proposition~\ref{prop:linear}, namely
$-\beta_0I\preceq\nabla^2\log p_0\preceq-\alpha_0I$ with $0<\alpha_0\le\beta_0<\infty$, then
Proposition~\ref{prop:linear} supplies \textup{(S)} on $[0,T]$ and
Theorem~\ref{thm:lagrangian} gives a unique RLF for
$\dot Z_s=-\beta(s)Z_s-\tfrac12\sigma(s)^2\nabla\log p_s(Z_s)$ transporting $p_0$ onto $p_s$.
\item[\textup{(iii)}] Condition \eqref{eq:p0hess} may \emph{not} be omitted: it is not
implied by the hypotheses of \textup{(i)}, and without it \textup{(S)} can fail ; see
Remark~\ref{rem:corfix}. But if $\supp p_0\subset B_R$ then for every $\delta\in(0,T)$, \textup{(S)}
holds on $[\delta,T]$ by Proposition~\ref{prop:earlystop} 
and Theorem~\ref{thm:lagrangian} applies there. The constant in $(S3)$ is
$O(R^2\varsigma_{\min}^{-4})$, the constant in $(S4)$ is
$O(\varsigma_{\min}^{-2})$, and the log-Hessian bound governing
$(S2)$ is
$O(\varsigma_{\min}^{-2}+R^2\varsigma_{\min}^{-4})$.
\end{enumerate}
\end{corollary}
\begin{figure}[ht]
\centering
\begin{tikzpicture}[
    node distance=2.2cm,
    >=Latex,
    box/.style={
        draw,
        rounded corners,
        align=center,
        minimum height=1.25cm,
        minimum width=3.5cm,
        inner sep=8pt
    },
    note/.style={
        align=center,
        font=\small,
        text width=3.6cm
    },
    arrow/.style={
        ->,
        thick
    }
]

\node[box] (mu0) {
    $X_0\sim \mu_0$\\[2pt]
     $\operatorname{supp}\mu_0\subset B_R$\\
    Possibly singular
};

\node[box, right=of mu0] (pdelta) {
    $t=\delta  $\\[2pt]
    $\operatorname{Law}(X_\delta)
      =p_\delta\,\mathcal L^d$\\
    $p_\delta>0$ and smooth
};

\node[box, right=of pdelta] (ps) {
    $t>\delta  $\\[2pt]
    $\operatorname{Law}(X_t)
      =p_t\,\mathcal L^d$\\
    $t\in[\delta,T]$
};

\draw[arrow] (mu0) -- node[above, align=center] {
    Diffusion\\
     $0\to \delta$
} (pdelta);

\draw[arrow] (pdelta) -- node[above, align=center] {
    PF-ODE\\
    $Z_{\delta,s}$
} (ps);

\node[note, below=1.15cm of pdelta] (score) {
    Proposition~5.3 gives
    \[
      \nabla\log p_t,
      \qquad t\in[\delta,T],
    \]
    together with \textup{(S1)}--\textup{(S4)}.
};

\node[note, below=1.15cm of ps] (transport) {
    Theorem~4.5, applied to the
    time-shifted problem, yields
    \[
      (Z_{\delta,s})_\#
      (p_\delta\mathcal L^d)
      =p_s\mathcal L^d .
    \]
};

\node[note, below=1.15cm of mu0] (noscore) {
    No density or score
    $\nabla\log\mu_0$ is required.
};

\draw[dashed, ->] (mu0) -- (noscore);
\draw[dashed, ->] (pdelta) -- (score);
\draw[dashed, ->] (ps) -- (transport);

\node[
    draw,
    dashed,
    rounded corners,
    fit=(pdelta)(ps)(score)(transport),
    inner sep=10pt,
    label={[font=\small]above:
        Early-stopped problem on $[\delta,T]$}
] {};

\end{tikzpicture}
\caption{
Diffusion regularizes
the possibly singular initial law by time $\delta$; the probability-flow
ODE is then applied only to the smooth density flow starting from
$p_\delta$.
}
\label{fig:early-stopping}
\end{figure}
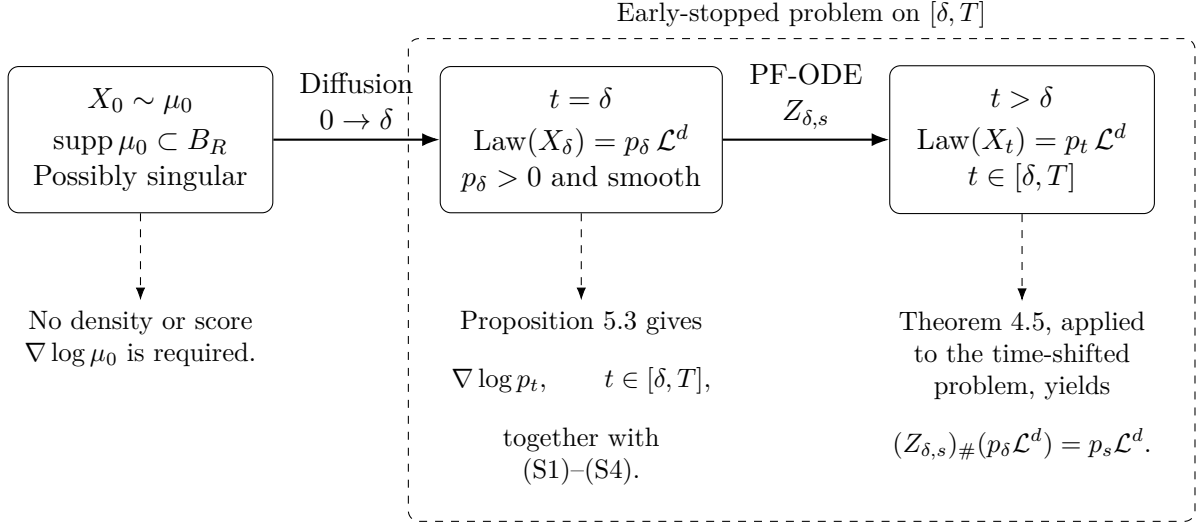
\begin{remark}
\label{rem:corfix}Corollary~\ref{cor:diffusion}(ii) requires \eqref{eq:p0hess}, which is \emph{strong
log-concavity} together with a bounded log-Hessian; it is a far stronger requirement than
belonging to $L^1\cap L^\infty$ with finite entropy and second moment. Condition
\eqref{eq:p0hess} constrains the second derivative of $\log p_0$ pointwise  and  forces $p_0$ to be strictly positive with sub-Gaussian tails.
A counterexample will be given in Section~\ref{sec:counterexample}.

\textup{(iii)} states that for $\delta>0$ the situation is  good but
the reason is \emph{not} that smoothing restores \eqref{eq:p0hess}. It does not: Gaussian smoothing of compactly supported multimodal data
need not be log-concave at moderate noise levels, so $p_\delta$ cannot in general be fed back
into Proposition~\ref{prop:linear}. Part \textup{(iii)} relies instead on the direct
verification of Proposition~\ref{prop:earlystop}, which uses only semiconvexity 
 together with compact support to control $\Cov(Y\mid X=x)$ from above. For the variance-exploding model, $\varsigma_{\min}^2\asymp\delta$,
so Proposition~\ref{prop:earlystop} gives the worst-case estimate
\[
\Lambda_\delta
=O(R^2\delta^{-2})
\]
for $(S3)$, while the $(S4)$ constant is $O(\delta^{-1})$.
This worst-case positive-part estimate need not be attained. For
example, the uniform density of Section~6 is log-concave, and Gaussian
convolution preserves log-concavity, so
$[\partial_x^2\log p_t]^+=0$ for every $t>0$ while the negative part grows like $t^{-1}$.
\end{remark}

\begin{corollary}[Time reversal]\label{cor:reverse}
Assume the hypotheses of Theorem~\ref{thm:lagrangian} \emph{and, in addition}, the
 one-sided bound
\begin{equation}\label{eq:divplus}
  \big[\Div v\big]^+\in L^1\big(0,T;L^\infty(\R^d)\big).
\end{equation}
Let
$Y_\tau:=X_{T-\tau}$, so that the reverse PF-ODE has velocity
$\tilde v(y,\tau):=-v(y,T-\tau)$. Then:
\begin{enumerate}
\item[\textup{(i)}] $\tilde v$ satisfies \textup{(R1)} and \textup{(R3)} with the same
constants as $v$, both being invariant under $v\mapsto-v$ and $s\mapsto T-s$.
\item[\textup{(ii)}] $\tilde v$ satisfies \textup{(R2)}, but \emph{not} with the same
constant: since $\Div\tilde v(\cdot,\tau)=-\Div v(\cdot,T-\tau)$ one has
$[\Div\tilde v]^-=[\Div v]^+$, so the relevant norm is
$\Theta_+$ of \eqref{eq:jac}, supplied by \eqref{eq:divplus} and in general unrelated to the
forward constant $\Theta_-$. The reverse RLF $\tilde Z$ exists and is unique, with
compressibility constant $e^{\Theta_+}$.
\item[\textup{(iii)}] Under \eqref{eq:divplus} the two-sided Jacobian bound \eqref{eq:jac}
holds, so $Z(s,\cdot)$ is essentially injective and
$\tilde Z(\tau,\cdot)=Z(T-\tau,\cdot)\circ Z(T,\cdot)^{-1}$ $\Leb$-a.e.
\item[\textup{(iv)}] $\tilde Z(\tau,\cdot)_\#\big(p_T\Leb\big)=p_{T-\tau}\Leb$ for every
$\tau\in[0,T]$; in particular the deterministic reverse sampler transports $p_T$ onto
$p_0$.
\end{enumerate}
\end{corollary}

Note that the forward and reverse directions are not symmetric.
Assumption \eqref{eq:divplus} is a genuine addition and cannot be read off from
Theorem~\ref{thm:lagrangian} which is deliberately one-sided. Reversing time
exchanges the roles of $[\Div v]^\pm$, so the reverse flow needs the bound the forward
theory never assumed.
Essential invertibility
of $Z$ and hence the identification of $\tilde Z$ with $Z^{-1}$ is a two-sided
statement and does not follow from the forward hypotheses; both directions require
\eqref{eq:divplus}. A sampler
that runs in one direction only needs a one-sided bound, whereas anything requiring
invertibility (exact likelihood evaluation by   change-of-variables
formula, encode-decode round trips) needs both.

\begin{remark}[Consistency with the reverse SDE]
The reverse SDE \cite{anderson1982,haussmann1986} and the reverse PF-ODE differ in the score coefficient,
$\sigma^2$ versus $\tfrac12\sigma^2$. The factor $\tfrac12$ is what turns
$\tfrac12\sigma^2\Delta p$ into a divergence, removing all randomness while preserving the probability flow
\eqref{eq:fp}. Both share the same marginal flow $(p_t)$ but the reverse SDE also reproduces the transition
statistics of the time-reversed diffusion.
\end{remark}

\section{Regularity of the score function}\label{sec:verify}

Theorem~\ref{thm:lagrangian} requires regularity assumptions --denoted (S)-- on the score function $s(t,x)$. In this section we provide sufficient conditions under which these regularity assumptions hold, in settings relevant for applications.
Alternative sufficient conditions for score regularity have recently been obtained by St\'ephanovitch \cite{Stephanovitch2025ScoreRegularity}.
We begin with a useful lemma:
\begin{lemma}[Log-Hessian bound under Gaussian smoothing]\label{lem:gauss}
Let $q_0$ be a probability density with $-\beta I\preceq\nabla^2\log q_0\preceq-\alpha I$,
$0<\alpha\le\beta<\infty$, and $q:=q_0*\mathcal N(0,\varsigma^2I)$. Then
\[
  -\frac{\beta}{1+\beta\varsigma^2}I\preceq\nabla^2\log q\preceq-\frac{\alpha}{1+\alpha\varsigma^2}I,
\]
so $\opnorm{\nabla^2\log q}\le\beta$ uniformly in $\varsigma>0$ and $|\Delta\log q|\le d\beta$.
\end{lemma}

\begin{proof}
Let $Y\sim q_0$, $X=Y+\varsigma\xi\sim q$. From
$\nabla\phi_\varsigma(x-y)=-\varsigma^{-2}(x-y)\phi_\varsigma(x-y)$ one gets
$\nabla\log q(x)=-\varsigma^{-2}(x-\E[Y\mid X=x])$, and with the Tweedie identity
$\nabla_x\E[Y\mid X=x]=\varsigma^{-2}\Sigma(x)$, $\Sigma(x):=\Cov(Y\mid X=x)$,
\begin{equation}
  \nabla^2\log q(x)=-\varsigma^{-2}I+\varsigma^{-4}\Sigma(x).\label{eq.hessian}
\end{equation}
The posterior has density $\propto e^{-U_x}$, $U_x(y)=-\log q_0(y)+\frac{|x-y|^2}{2\varsigma^2}$,
so $(\alpha+\varsigma^{-2})I\preceq\nabla^2U_x\preceq(\beta+\varsigma^{-2})I$. The Brascamp--Lieb inequality
\cite{BL1976} gives $\Sigma\preceq(\alpha+\varsigma^{-2})^{-1}I$; the Cram\'er--Rao bound gives
$$\Sigma\succeq(\E\nabla^2U_x)^{-1}\succeq(\beta+\varsigma^{-2})^{-1}I.$$ Inserting into
\eqref{eq.hessian} and simplifying gives the claim.
\end{proof}

Let us first examine the case where the SDE has linear drift: this is the case for so-called 'variance-preserving' (VP) and 'variance-exploding' (VE) forward diffusions used in generative diffusion models \cite{hanxu2024,Song2021,Song2021mle}:
\begin{proposition}[Linear drift: VP and VE SDEs]\label{prop:linear}
Assume \textup{(E1)}, $f(x,s)=-\beta(s)x$ with $\beta\in L^\infty(0,T)$, and the two-sided
log-Hessian condition
\begin{equation}\label{eq:p0hess}
  -\beta_0I\;\preceq\;\nabla^2\log p_0\;\preceq\;-\alpha_0I,
  \qquad 0<\alpha_0\le\beta_0<\infty .
\end{equation}
 Then
\textup{(S)} holds: with $m_s:=e^{-\int_0^s\beta}\in[m_{\min},m_{\max}]\subset(0,\infty)$,
\[
  \opnorm{\nabla^2\log p_s}\le\frac{\beta_0}{m_s^2}\le\frac{\beta_0}{m_{\min}^2},
  \qquad|\Delta\log p_s|\le\frac{d\beta_0}{m_{\min}^2},
\]
and $|\nabla\log p_s(x)|\le A_*(1+|x|)$ with
$A_*=\max\{\sup_{s\le T}|\nabla\log p_s(0)|,\ \beta_0m_{\min}^{-2}\}<\infty$.
\end{proposition}

\begin{proof}
Variation of constants gives $X_s\stackrel{d}{=}m_sX_0+\varsigma_s\xi$ with
$\varsigma_s^2=m_s^2\int_0^s\sigma(r)^2m_r^{-2}\dd r>0$ for $s>0$, the noise being isotropic
because $\sigma$ is scalar and $x$-independent. Hence
$p_s=p_0^{(s)}*\mathcal N(0,\varsigma_s^2I)$ with $p_0^{(s)}(x)=m_s^{-d}p_0(x/m_s)$, and
$\nabla^2\log p_0^{(s)}(x)=m_s^{-2}\nabla^2\log p_0(x/m_s)$, so
$-\beta_0m_s^{-2}I\preceq\nabla^2\log p_0^{(s)}\preceq-\alpha_0m_s^{-2}I$.
Lemma~\ref{lem:gauss} with $\alpha=\alpha_0m_s^{-2}$, $\beta=\beta_0m_s^{-2}$,
$\varsigma=\varsigma_s$ gives the Hessian bounds, hence (S3) with constant $\Lambda$; (S2)
follows since $\nabla\log p_s$ is smooth with bounded derivative. The Hessian bound makes
$\nabla\log p_s$ globally Lipschitz, and $s\mapsto\nabla\log p_s(0)$ is continuous on the
compact $[0,T]$, giving (S4); (S4$'$) is then automatic, or alternatively follows from
Lemma~\ref{lem:fisher}. For (S1): positivity and smoothness are immediate for $s>0$;
$\|p_s\|_\infty\le m_{\min}^{-d}\|p_0\|_\infty$; and
$\int|x|^2p_s=m_s^2\int|x|^2p_0+d\varsigma_s^2$ is bounded on $[0,T]$.
\end{proof}
Brigati and Pedrotti \cite{BrigatiPedrotti2025} derive quantitative bounds on $\nabla^2\log p_t $
 for Gaussian convolutions of strongly log-concave measures with Lipschitz perturbations. Their estimates similarly imply spatial Lipschitz regularity of the score and thus classical well-posedness of the probability-flow ODE. Our objective is different: rather than deriving additional sufficient conditions for Lipschitz regularity, we investigate the probability-flow ODE under substantially weaker Sobolev and BV assumptions on the score.

The following proposition deals with the frequently encountered situation where the initial distribution $\mu_0$ may be concentrated on a lower-dimensional compact submanifold (so in particular not absolutely continuous): 
\begin{proposition}[Early-stopping for compactly supported data]
\label{prop:earlystop} Assume \textup{(E1)}.
 Let $X=(X_t)_{t\in[0,T]}$ be the solution of the linear SDE with drift $f(x,s)=-\beta(s)x,\ \beta\in L^\infty(0,T)$:
 $$ dX_t= -\beta(t)X_t dt +\sigma(t) dW_t,\qquad X_0\sim \mu_0$$
where $\mu_0\in{\cal P}(\mathbb{R}^d)$ is a
probability  measure on $\mathbb{R}^d$ with compact support $\supp \mu_0\subset B_R$. No absolute continuity, regularity, or log-concavity of $\mu_0$  is assumed. 
Let $0<\delta<T$ and  $m_s,\varsigma_s$ as in
Proposition~\ref{prop:linear},
\[
\varsigma_{\min}^2:=\inf_{s\in[\delta,T]}\varsigma_s^2>0,\qquad
  m_{\max}:=\sup_{s\le T}m_s<\infty .
\]
Then the density $p_t$ of $X_t$ satisfies \textup{(S)}  on $[\delta,T]$, with the explicit constants
\begin{align}
\bigl\|[\Delta\log p_s]^+\bigr\|_{L^\infty}
\le
d\left(
\frac{m_s^2R^2}{\varsigma_s^4}
-
\frac1{\varsigma_s^2}
\right)^+
\le
d\,m_{\max}^2R^2\varsigma_{\min}^{-4}, &&\text{giving \textup{(S3)}},
  \label{eq:esS3}\\
  \|\nabla^2\log p_s\|_{\mathrm{op}}
\le
\varsigma_s^{-2}
+m_s^2R^2\varsigma_s^{-4}
\le
\varsigma_{\min}^{-2}
+
m_{\max}^2R^2\varsigma_{\min}^{-4}
  &&\text{giving \textup{(S2)}},
  \label{eq:esS2}\\
  |\nabla\log p_s(x)|
\le
\varsigma_{\min}^{-2}
\max\{1,m_{\max}R\}(1+|x|),
  &&\text{giving \textup{(S4)}} .
  \label{eq:esS4}
\end{align}
In addition,
\[
\nabla^2\log p_s\ge-\varsigma_s^{-2}I,
\qquad
\bigl\|[\Delta\log p_s]^-\bigr\|_{L^\infty}
\le d\varsigma_{\min}^{-2}.
\]
\end{proposition}

\begin{proof}
Variation of constants gives $p_s\Leb  =\mu_0^{(s)}*\mathcal N(0,\varsigma_s^2I)$, where $\mu_0^{(s)}$
is the pushforward of $\mu_0$ under $y\mapsto m_sy$, supported in $B_{m_sR}$; and
$\varsigma_s^2>0$ for $s>0$ by \textup{(E1)}. Thus $p_s$ is smooth and strictly
positive for $s>0$, which with $\|p_s\|_{L^\infty}\le(2\pi\varsigma_{\min}^2)^{-d/2}$,
$\int|x|^2p_s(x)dx=m_s^2\int|x|^2\mu_0(dx)+d\varsigma_s^2$ and narrow continuity gives \textup{(S1)}.

Write $Y\sim \mu_0^{(s)}$ and $X=Y+\varsigma_s\xi\sim p_s$. The Tweedie identities used in
Lemma~\ref{lem:gauss} read
\begin{equation}\label{eq:tweedie}
  \nabla\log p_s(x)=-\frac{x-\E[Y\mid X=x]}{\varsigma_s^2},
  \qquad
  \nabla^2\log p_s(x)=-\frac{1}{\varsigma_s^2}I+\frac{1}{\varsigma_s^4}\Sigma(x),
  \quad \Sigma(x)=\Cov(Y\mid X=x).
\end{equation}
Since $\Sigma_s(x)\succeq0$, the identity \eqref{eq:tweedie} gives the
semiconvexity estimate
\begin{equation}
    \nabla^2\log p_s(x)\succeq-\varsigma_s^{-2}I.\label{eq.lowerbound}
\end{equation}
This controls the negative part of $\Delta\log p_s$, but the
one-sided condition $(S3)$ requires an upper bound on its positive
part.
The conditional law of $Y$ given $X=x$ is supported in
$B_{m_sR}$. Hence, for every unit vector $\xi$,
\[
\xi^\top\Sigma_s(x)\xi
=
\operatorname{Var}(\xi\cdot Y\mid X=x)
\le
\mathbb E[(\xi\cdot Y)^2\mid X=x]
\le m_s^2R^2,
\]
and therefore
\[
0\preceq\Sigma_s(x)\preceq m_s^2R^2I.
\]
Taking traces in \eqref{eq:tweedie} yields
\[
\Delta\log p_s(x)
=
-\frac d{\varsigma_s^2}
+
\frac{\operatorname{tr}\Sigma_s(x)}{\varsigma_s^4}
\le
-\frac d{\varsigma_s^2}
+
\frac{d\,m_s^2R^2}{\varsigma_s^4}.
\]
Consequently,
\[
[\Delta\log p_s(x)]^+
\le
d\left(
\frac{m_s^2R^2}{\varsigma_s^4}
-
\frac1{\varsigma_s^2}
\right)^+
\le
d\,m_{\max}^2R^2\varsigma_{\min}^{-4},
\]
which proves $(S3)$.
Combining
\[
-\varsigma_s^{-2}I
\preceq
\nabla^2\log p_s(x)
\preceq
\left(
m_s^2R^2\varsigma_s^{-4}
-\varsigma_s^{-2}
\right)I
\]
gives
\[
\|\nabla^2\log p_s\|_{\mathrm{op}}
\le
\varsigma_s^{-2}
+
m_s^2R^2\varsigma_s^{-4},
\]
and hence $(S2)$. Finally,
$\mathbb E[Y\mid X=x]\in B_{m_sR}$, so the first identity in
\eqref{eq:tweedie} gives
\[
|\nabla\log p_s(x)|
\le
\varsigma_s^{-2}(|x|+m_sR),
\]
which yields $(S4)$.
\end{proof}

\begin{proposition}[Nonlinear gradient drift]\label{prop:gradient}
Let $\sigma\equiv\sigma_0$ be constant and $f=-\nabla V$ with $V\in C^3(\R^d)$,
$\inf V>-\infty$, and
\[
  \nabla^2V\succeq-\kappa I,\qquad \|\nabla^2V\|_\infty\le L_2,\qquad
  \|\nabla^3V\|_\infty\le L_3
\]
for some $\kappa\in\R$ and $L_2,L_3<\infty$. Assume in addition the \emph{confinement
condition}
\begin{equation}\label{eq:confine}
  \big\langle\nabla V(x),x\big\rangle\;\ge\;a|x|^2-b
  \qquad\text{for all }x\in\R^d,
\end{equation}
with constants $a>0$ and $b\ge0$. Let $\mu:=e^{-2V/\sigma_0^2}$ and suppose $p_0=\mu h_0$
with $0<c_0\le h_0\le C_0<\infty$ and
$\|\nabla h_0\|_\infty,\|\nabla^2h_0\|_\infty<\infty$. Then \textup{(S)} holds, with
$\opnorm{\nabla^2\log p_s}\le M$ for all $s\in[0,T]$, where $\kappa_+:=\max(\kappa,0)$ and
\[
  M=\frac{2L_2}{\sigma_0^2}
   +\frac{e^{2\kappa_+T}}{c_0}\big(\|\nabla^2h_0\|_\infty+L_3T\|\nabla h_0\|_\infty\big)
   +\frac{e^{2\kappa_+T}}{c_0^2}\|\nabla h_0\|_\infty^2 ,
\]
and moreover
\begin{equation}\label{eq:confmoment}
  \sup_{s\in[0,T]}\int_{\R^d}|x|^2p_s(x)\,\dd x
  \;\le\;\int_{\R^d}|x|^2p_0(x)\,\dd x+\frac{2b+d\ \sigma_0^2}{2a}\;<\;\infty .
\end{equation}
\end{proposition}
The proof is given in Appendix \ref{app.5.4}

\begin{remark}[Condition \eqref{eq:confine} cannot be dropped: the Cauchy example]
\label{rem:cauchy}
Without \eqref{eq:confine} the conclusion is false, and it fails in the moment clause of
\textup{(S1)} \emph{alone}. Take $d=1$, $\sigma_0^2=2$ and
\[
  V(x)=\log\big(1+x^2\big),\qquad h_0\equiv\frac1\pi .
\]
All the remaining assumptions hold. $V$ is smooth with $\inf V=0$;
$V''(x)=\tfrac{2(1-x^2)}{(1+x^2)^2}$ has $\min V''=-\tfrac14$ at $x=\pm\sqrt3$ and
$\|V''\|_\infty=2$, so $\kappa=\tfrac14$ and $L_2=2$; $V'''$ is bounded
($L_3\approx2.91$); and $h_0$ is constant, so $c_0=C_0=\tfrac1\pi$ and
$\nabla h_0=\nabla^2h_0=0$. Here $\mu(x)=e^{-2V/\sigma_0^2}=(1+x^2)^{-1}$ with
$\int\mu=\pi$, so
\[
  p_0(x)=\mu(x) h_0=\frac{1}{\pi(1+x^2)}
\]
is the standard \emph{Cauchy} density. Since $h_0$ is constant and $P_s\mathbf 1=\mathbf 1$,
we get $h_s\equiv h_0$ and therefore $p_s\equiv p_0$ for every $s$: the flow is stationary
at the invariant measure. Consequently $\int|x|\,p_t\,\dd x=\infty$  and the
moment requirement in \textup{(S1)} fails.

\textup{(S2)} holds;
$\partial_x^2\log p_s=\tfrac{2(x^2-1)}{(1+x^2)^2}$ is bounded by $2$, giving \textup{(S3)};
and the score $\partial_x\log p_0(x)=-\tfrac{2x}{1+x^2}$ is \emph{bounded} by $1$, so
\textup{(S4)} holds with room to spare. Even the quantitative conclusion is correct: the proposition predicts $M=2L_2/\sigma_0^2=2$, and indeed
$\sup_x|\partial_x^2\log p_s|=2$.

The issue is that the assumptions on $V$ bound $\nabla^2V$ from above
and below but impose no \emph{growth} condition, so $\mu=e^{-2V/\sigma_0^2}$ may have merely
power-law tails; and assumptions on $h_0$ constrain only the \emph{ratio} $p_0/\mu$,
never $\mu$ itself. Note also that $p_0$ is not log-concave here ---
$\partial_x^2\log p_0(3)=\tfrac{4}{25}>0$ --- which is exactly what distinguishes this
setting from Proposition~\ref{prop:linear}, where the two-sided hypothesis
$\nabla^2\log p_0\preceq-\alpha_0I$ with $\alpha_0>0$ is strong log-concavity and
leads to sub-Gaussian tails. Condition
\eqref{eq:confine} excludes the example: here
$\langle V'(x),x\rangle=\tfrac{2x^2}{1+x^2}\le2$ is bounded, so \eqref{eq:confine} fails for
every $a>0$.

An alternative to \eqref{eq:confine}, is to require
$\int|x|^2\mu_0(\dd x)<\infty$ which  then leads to
\eqref{eq:confmoment}. 
\end{remark}

\section{A counterexample}\label{sec:counterexample}
 Theorem ~\ref{thm:lagrangian}, which gives condition for a well-posed ODE (Lagrandian) flow  requires much more assumptions than
Theorem~\ref{thm:eulerian}, which governs the probability flow generated by the SDE \eqref{eq:sde}.
The
gap between them is not an artefact: it reflects a fundamental fact that much more regularity is needed for the PF-ODE to replicate the same probability flow as the SDE: one {\it should not}  assume that because the latter is well-posed, the former will generate the same flow.

To illustrate this gap,
we provide in this section an example 
of well-posed 
SDE whose density flow $(p_t)$ is uniquely determined by
Theorem~\ref{thm:eulerian}, for which the continuity
equation \eqref{eq:cont} is \emph{not} well posed and  admits no regular Lagrangian flow.

Consider standard Brownian motion started from the uniform law on $[-1,1]$:
\begin{equation}\label{eq:ceX}
  d=1,\qquad f\equiv0,\qquad \sigma\equiv1,\qquad p_0=\tfrac12\,\mathbf 1_{[-1,1]}.\end{equation}
Assumptions \textup{(E1)}, \textup{(D$_{\mathrm L}$)} and
\textup{(I1)} hold --- the drift conditions vacuously, since $f\equiv0$ --- and
$p_0\in L^1\cap L^\infty$ with
$\|p_0\|_\infty=\tfrac12$, $\int x^2p_0\,\dd x=\tfrac13$ and
$\Ent(p_0)=-\log 2$. Moreover $\nabla p\in L^2([0,T];L^2)$. Hence
Theorem~\ref{thm:eulerian} applies: $p_t=p_0*\mathcal N(0,t)$ is the unique weak solution
of \eqref{eq:fp} in $\mathcal X$ with initial condition  $p_0$, and the SDE is strongly well posed with
pathwise uniqueness.

Only the energy bound needs comment. Away from the two jump points of $p_0$ the derivative
$\partial_xp_t$ is uniformly bounded on $[\eta,T]$ for each $\eta>0$; near $x=\pm1$ the
initial layer has width $O(\sqrt t)$ and $|\partial_xp_t|=O(t^{-1/2})$ there, so
$\int_{\R}|\partial_xp_t|^2\dd x=O(t^{-1/2})$ and
$\int_0^T\!\int|\partial_xp_t|^2\,\dd x\,\dd t<\infty$.

 Denote by $\Phi,\phi$  the standard normal distribution and density. Convolution of
\eqref{eq:ceX} with $\mathcal N(0,t)$ gives, for $t>0$,
\begin{align}
  p_t(x)&=\tfrac14\Big[\operatorname{erf}\Big(\tfrac{1-x}{\sqrt{2t}}\Big)
                      +\operatorname{erf}\Big(\tfrac{1+x}{\sqrt{2t}}\Big)\Big],
  \label{eq:cep}\\[2pt]
  \partial_xp_t(x)&=\frac{1}{2\sqrt{2\pi t}}
        \Big[e^{-(x+1)^2/2t}-e^{-(x-1)^2/2t}\Big],
  \label{eq:cedp}\\[2pt]
  F_t(x)&:=\int_{-\infty}^xp_t
   =\frac{\sqrt t}{2}\Big[G\Big(\tfrac{x+1}{\sqrt t}\Big)-G\Big(\tfrac{x-1}{\sqrt t}\Big)\Big],
   \qquad G(u):=u\Phi(u)+\phi(u).
  \label{eq:ceF}
\end{align}
Since $f\equiv0$ and $\sigma\equiv1$, the probability-flow velocity of \eqref{eq:cont} is
$v=-\tfrac12\partial_x\log p_t$, i.e.
\begin{equation}\label{eq:cev}
  v(x,t)\;=\;\frac{1}{\sqrt{2\pi t}}\;
  \frac{e^{-(x-1)^2/2t}-e^{-(x+1)^2/2t}}
       {\operatorname{erf}\!\big(\tfrac{1-x}{\sqrt{2t}}\big)
       +\operatorname{erf}\!\big(\tfrac{1+x}{\sqrt{2t}}\big)},
\end{equation}
and the PF-ODE \eqref{eq:pfode} reads $\dot Z_t=v(Z_t,t)$ with $v$ given by \eqref{eq:cev}.
The field is odd, real-analytic and strictly positive on $x>0$ for each $t>0$; it vanishes
at $x=0$ by symmetry.

\begin{lemma}[Behaviour outside the initial support]\label{lem:ce-asymp}
Fix $x>1$. As $t\downarrow0$,
\[
  p_t(x)=\frac{\sqrt t}{2\sqrt{2\pi}\,(x-1)}\,e^{-(x-1)^2/2t}\big(1+O(t)\big),
\]
and consequently
\[
  \partial_x\log p_t(x)=-\frac{x-1}{t}-\frac{1}{x-1}+O(t),
  \qquad
  \partial_x^2\log p_t(x)=-\frac1t+\frac{1}{(x-1)^2}+O(t),
\]
so that
\begin{equation}\label{eq:cevasym}
  v(x,t)=\frac{x-1}{2t}+\frac{1}{2(x-1)}+O(t)\qquad(t\downarrow0).
\end{equation}
The corresponding statements for $x<-1$ follow by symmetry. In particular, for any compact set
$K\subset\{|x|>1\}$,
\begin{equation}\label{eq:ceW11}
  \int_K\big|\partial_x\log p_t\big|\,\dd x=\frac{1}{t}\int_K\big(|x|-1\big)\dd x+O(1),
  \qquad
  \int_K\big|\partial_x^2\log p_t\big|\,\dd x=\frac{|K|}{t}+O(1).
\end{equation}
\end{lemma}

\begin{proof}
For $x>1$ rewrite \eqref{eq:cep} as
$p_t(x)=\tfrac14\big[\operatorname{erfc}\big(\tfrac{x-1}{\sqrt{2t}}\big)
-\operatorname{erfc}\big(\tfrac{x+1}{\sqrt{2t}}\big)\big]$ and insert the standard expansion
$\operatorname{erfc}(u)=\frac{e^{-u^2}}{u\sqrt\pi}\big(1-\tfrac{1}{2u^2}+O(u^{-4})\big)$
with $u=\tfrac{x-1}{\sqrt{2t}}\to\infty$; the second term is smaller by a factor
$e^{-2x/t}$ and is absorbed in the error. Taking logarithms,
$\log p_t(x)=-\tfrac{(x-1)^2}{2t}-\log(x-1)+\tfrac12\log t+c+O(t)$, and differentiating in
$x$ gives the stated expansions. Finally $v=-\tfrac12\partial_x\log p_t$.
\end{proof}

\begin{remark}[Consistency with \eqref{eq.lowerbound}]\label{rem:ce-tweedie}
Lemma~\ref{lem:ce-asymp} gives $\partial_x^2\log p_t\to-1/t$ outside $[-1,1]$, which
saturates the universal lower bound $\nabla^2\log p_t\succeq-\tfrac1tI$ of
\eqref{eq.lowerbound}. The example is therefore extremal for the Gaussian-smoothing estimate
of Lemma~\ref{lem:gauss}: it is exactly the case in which the posterior covariance
$\Sigma(x)$ collapses, because conditionally on a far-field observation the initial point
is pinned to the support edge.
\end{remark}

\subsection{Solution of the PF-ODE: the quantile flow}

The following identity is general and of independent interest; it gives closed-form
solutions of the PF-ODE in one dimension for \emph{any} initial density.

\begin{lemma}[One-dimensional PF-ODE $=$ monotone rearrangement]\label{lem:quantile}
Let $d=1$ and let $(p_t)$ solve \eqref{eq:fp} with $\int_{-\infty}^{x}p_t$ finite, and set
$F_t(x):=\int_{-\infty}^xp_t$. Assume $p_t>0$ and $p_tv\in L^1_\loc$ with vanishing flux at
$-\infty$. Then
\begin{equation}\label{eq:ceflux}
  \partial_tF_t(x)=-p_t(x)\,v(x,t),
\end{equation}
and consequently $t\mapsto F_t(Z_t)$ is constant along every absolutely continuous solution
of $\dot Z_t=v(Z_t,t)$. Hence, whenever the right-hand side is defined,
\begin{equation}\label{eq:cequantile}
  Z(t,x)=F_t^{-1}\big(F_0(x)\big).
\end{equation}
\end{lemma}

\begin{proof}
Integrating $\partial_tp_t+\partial_x(p_tv)=0$ over $(-\infty,x)$ gives \eqref{eq:ceflux}.
Then, along a solution,
$\frac{\dd}{\dd t}F_t(Z_t)=\partial_tF_t(Z_t)+p_t(Z_t)\dot Z_t
=-p_t(Z_t)v(Z_t,t)+p_t(Z_t)v(Z_t,t)=0$.
Thus the \emph{quantile level} $q:=F_t(Z_t)$ is a first integral, and inverting the strictly
increasing $F_t$ yields \eqref{eq:cequantile}.
\end{proof}

For the case \eqref{eq:ceX}, $F_0(x)=\tfrac{x+1}{2}$ on $[-1,1]$, $F_0\equiv0$ on
$(-\infty,-1]$ and $F_0\equiv1$ on $[1,\infty)$, while $F_t$ given by \eqref{eq:ceF} is a
strictly increasing real-analytic bijection
\[
F_t:\R\longrightarrow(0,1)\qquad\text{for every }t>0 .
\]

\subsection{Failure of Lagrangian well-posedness}

\begin{theorem}[The PF continuity equation is not well posed]\label{thm:ce-fail}
Let $v$ be given by \eqref{eq:cev} and $T>0$. 
\begin{enumerate}
\item[\textup{(i)}] $v\notin L^1_\loc\big([0,T]\times\R\big)$. The weak
formulation of Definition~\ref{def:class} is not meaningful for general
$u\in L^\infty((0,T);L^1\cap L^\infty)$, and assumptions \textup{(R1)}, \textup{(R3)} of
Theorem~\ref{thm:DL} fail.
\item[\textup{(ii)}] For every $x_0$ with $|x_0|\ge1$ there is \emph{no} absolutely
continuous $Z:[0,T]\to\R$ with $Z(0)=x_0$ and $Z(t)=x_0+\int_0^tv(Z(s),s)\,\dd s$.
Consequently no regular Lagrangian flow exists: condition \textup{(i)} of
Definition~\ref{def:rlf} fails on the set $\{|x|>1\}$.
\item[\textup{(iii)}] For every $x\in(-1,1)$ there is exactly one such solution, namely
\eqref{eq:cequantile}, and\\ $Z(t,\cdot):(-1,1)\to\R$ is a smooth increasing bijection.
\item[\textup{(iv)}] The Cauchy problem fails to be solvable
for any initial condition $u_0\in L^1\cap L^\infty$  with
$u_0=0$ a.e.\ on $[-1,1]$: there is no
narrowly continuous family of probability measures $(\mu_t)_{t\in[0,T]}$ with
$\mu_0=u_0\Leb$ solving \eqref{eq:cont} and satisfying the integrability
hypothesis of Theorem~\ref{thm:super}.  
\end{enumerate}
\end{theorem}

\begin{remark}\label{rem:ce-delta}
For every $\delta>0$ the field $v$ is smooth on
$\mathbb R\times[\delta,T]$ and, because $p_0$ is supported in
$[-1,1]$, the first Tweedie identity in~\eqref{eq:tweedie}
gives
\[
|\partial_x\log p_t(x)|
\le \frac{|x|+1}{t}
\le \frac{|x|+1}{\delta},
\]
since $\mathbb E[Y|X=x]\in[-1,1]$.
Hence (S1)--(S4) hold on $[\delta,T]$ and
Theorem~\ref{thm:lagrangian} applies, with constants of order $1/\delta$.
The obstruction is
the \emph{boundary of the support}, not lack of regularity.
\end{remark}
\begin{proof}
\emph{(i)} By \eqref{eq:cevasym}, for $1<a<b$ and $x\in[a,b]$ we have
$v(x,t)\ge\frac{a-1}{2t}(1+o(1))$ as $t\downarrow0$, hence
\[
  \int_0^T\!\!\int_a^b|v(x,t)|\,\dd x\,\dd t
  \;\ge\;\frac{(b-a)(a-1)}{4}\int_0^{t_0}\frac{\dd t}{t}\;=\;+\infty
\]
for $t_0$ small enough. Since (R1) and (R3) both presuppose local integrability in
space-time, both fail.

\emph{(ii)} Suppose $Z$ were such a solution with $Z(0)=x_0$, $|x_0|\ge1$. Since $p_t>0$
and $v$ is real-analytic on $\R\times(0,T]$, Lemma~\ref{lem:quantile} applies on every
$[\eta,T]$ and gives a constant $q\in(0,1)$ with $F_t(Z(t))=q$ for all $t\in(0,T]$. Now
$p_t\to p_0$ in $L^1$ as $t\downarrow0$, so $\|F_t-F_0\|_\infty\le\|p_t-p_0\|_{L^1}\to0$;
and $Z$ is continuous at $0$ with $Z(0)=x_0$. Therefore
\[
  q=\lim_{t\downarrow0}F_t\big(Z(t)\big)=F_0(x_0)\in\{0,1\},
\]
contradicting $q\in(0,1)$. Hence no such $Z$ exists. As $\{|x|>1\}$ has infinite measure,
Definition~\ref{def:rlf}(i) --- which demands trajectories for $\Leb$-a.e.\ $x\in\R$ ---
cannot hold.

\emph{(iii)} For $x\in(-1,1)$ set $q:=F_0(x)=\tfrac{x+1}{2}\in(0,1)$ and
$Z(t,x):=F_t^{-1}(q)$, which is well defined since $F_t:\R\to(0,1)$ is an increasing
bijection. Differentiating $F_t(Z)=q$ and using \eqref{eq:ceflux} gives $\dot Z=v(Z,t)$,
and $Z(t,x)\to F_0^{-1}(q)=x$ as $t\downarrow0$ by the uniform convergence $F_t\to F_0$
together with strict monotonicity, so the initial condition is attained. Uniqueness holds
because $v$ is locally Lipschitz in $x$ on $\R\times(0,T]$, so two solutions agreeing at
some $t>0$ agree throughout; and the first integral $q$ is determined by $x$. Analyticity
and monotonicity of $x\mapsto Z(t,x)$ follow from those of $F_0$ and $F_t^{-1}$, with
\begin{equation}\label{eq:cejac}
  \partial_xZ(t,x)=\frac{p_0(x)}{p_t\big(Z(t,x)\big)}=\frac{1}{2\,p_t(Z(t,x))}>0 .
\end{equation}
Surjectivity onto $\R$ holds because $F_0$ maps $(-1,1)$ onto $(0,1)$ and $F_t^{-1}$ maps
$(0,1)$ onto $\R$.

\emph{(iv)} Suppose such a family $(\mu_t)$ existed. By Theorem~\ref{thm:super} there is
$\eta\in\mathcal P(C([0,T];\R))$ concentrated on absolutely continuous integral curves of
$v$ with $(e_0)_\#\eta=\mu_0=u_0\Leb$. Since $u_0$ vanishes a.e.\ on $[-1,1]$,
$\eta$-a.e.\ curve $\gamma$ satisfies $|\gamma(0)|\ge1$; but by \textup{(ii)} no integral
curve has this property. Hence $\eta$ is the zero measure, contradicting
$\eta(C([0,T];\R))=1$.
\end{proof}
This example \eqref{eq:ceX} separates assumption \textup{(S4)} from \textup{(S4$'$)}:
\begin{enumerate}
\item[\textup{(a)}] \textup{(S2)} and \textup{(S4)} fail at
the same rate. Each is an unweighted local norm of the score integrated in time, and by
Lemma~\ref{lem:ce-asymp} every one of them diverges like $t^{-1}$ as $t\downarrow0$:
\begin{itemize}
\item[--] \textup{(S4)}: for fixed $x>1$, $|\partial_x\log p_t(x)|\ge\tfrac{x-1}{t}(1+o(1))$,
so no admissible $A$ lies in $L^1(0,T)$;
\item[--] \textup{(S2)}: for any compact $K\subset\{|x|>1\}$,
\begin{eqnarray*}
    \big\|\nabla\log p_t\big\|_{W^{1,1}(K)}
  =\int_K\!\big|\partial_x\log p_t\big|+\int_K\!\big|\partial_x^2\log p_t\big|
  \;=\;\frac{c_K}{t}\,\big(1+o(1)\big),
  \\
  c_K:=\int_K\!\big(|x|-1\big)\dd x+|K|>0,  
\end{eqnarray*}
so $t\mapsto\|\nabla\log p_t\|_{W^{1,1}(K)}$ is not in $L^1(0,T)$ and
$\nabla\log p_\cdot\notin L^1\big([0,T];W^{1,1}_\loc\big)$.
\end{itemize}
Smoothness of $\nabla\log p_t$   for each $t>0$ is not sufficient: \textup{(S2)} requires
integrability  \emph{in time}.
\item[\textup{(b)}] By contrast \textup{(S3)} and \textup{(S4$'$)} hold, together with \textup{(S1)} for
$t>0$. Indeed $p_t>0$ is real-analytic for $t>0$, $\|p_t\|_\infty\le\tfrac12$,
$\int x^2p_t=\tfrac13+t$, and by Lemma~\ref{lem:fisher} with $f\equiv0$, $\sigma\equiv1$,
\[
  \tfrac12\int_0^T\Fish(p_t)\,\dd t=\Ent(p_0)-\Ent(p_T)\le-\log2+C\big(1+\tfrac13+T\big)<\infty .
\]
The positivity clause of \textup{(S1)} fails at the single time $t=0$, where
$\operatorname{supp}p_0=[-1,1]$; this is immaterial, every other clause of \textup{(S)} being
an $L^1$-in-time condition.
\end{enumerate}
Hence Theorem~\ref{thm:lagrangian} is not applicable, but
\textup{(S4$'$)} is satisfied. We are thus in the situation described in Remark \ref{rem.iv}.

\subsection{Numerical illustration}\label{sec:ce-numerics}

\begin{figure}[htbp]
  \centering
  \includegraphics[width=\textwidth]{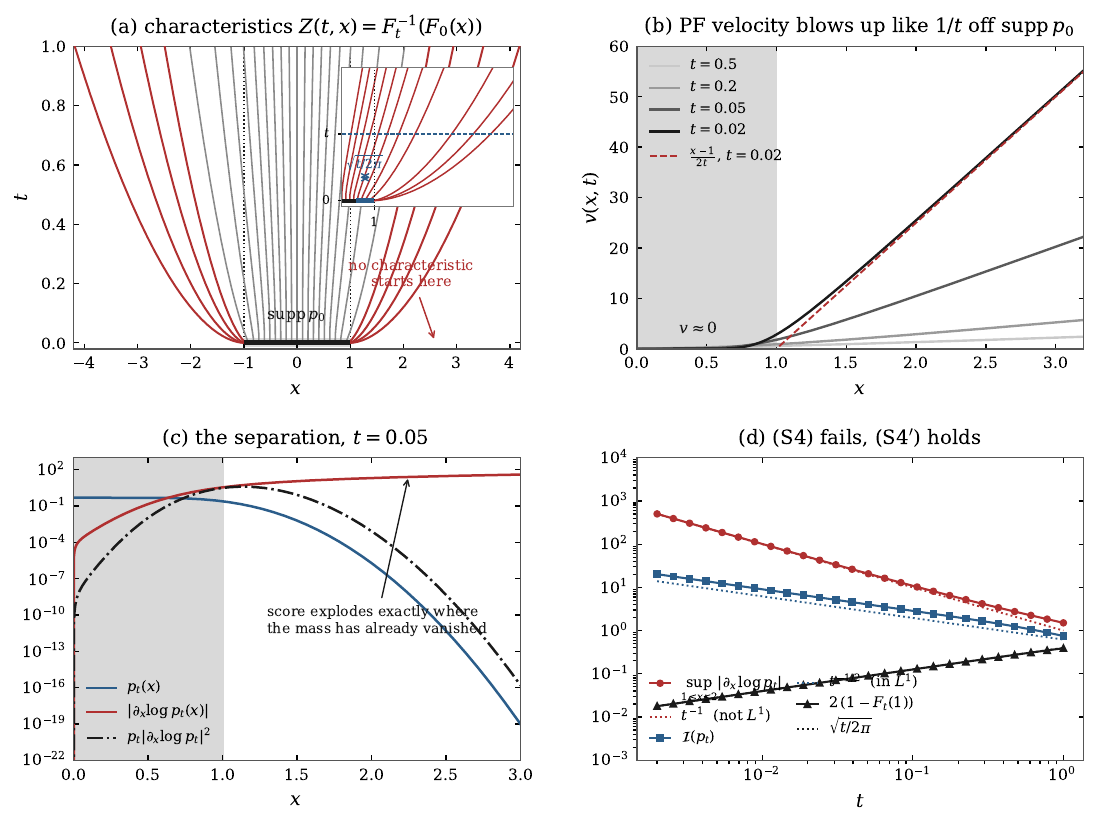}
  \caption{Marginal flow for Brownian motion with uniform initialisation 
  $p_0=\tfrac12\mathbf 1_{[-1,1]}$.\\
  \textbf{(a)} Characteristics $Z(t,x)=F_t^{-1}(F_0(x))$ of the PF-ODE. Every curve
  originates in $\operatorname{supp}p_0=[-1,1]$ (thick bar); the region $|x|>1$ at $t=0$
  emits none, which is Theorem~\ref{thm:ce-fail}(ii) and the reason no regular Lagrangian
  flow exists. Red curves are the near-edge quantiles. \emph{Inset:} the points lying
  outside $[-1,1]$ at time $t$ are the images of an interval of length
  $2(1-F_t(1))\sim\sqrt{t/2\pi}$ --- small, but of \emph{positive} measure, so the
  transport identity itself does not fail.
  \textbf{(b)} The velocity $v(x,t)$ is exponentially small on $[-1,1]$ and diverges like
  $(x-1)/2t$ off the support (dashed, $t=0.02$), giving
  $v\notin L^1_\loc$, Theorem~\ref{thm:ce-fail}(i).
  \textbf{(c)} At $t=0.05$: $|\partial_x\log p_t|$ grows without bound exactly where $p_t$
  has already decayed by twenty orders of magnitude, so the Fisher integrand
  $p_t|\partial_x\log p_t|^2$ stays negligible there.
  \textbf{(d)} The sup-norm of the score diverges   $\sim t^{-1}$ and is not integrable, so
$(S4)$ fails; the local $W^{1,1}$ norm in $(S2)$ also diverges at the
same rate. By contrast, $(S3)$ holds and $I(p_t)\sim t^{-1/2}$  is integrable, so $(S4')$
holds.}
  \label{fig:ce}
\end{figure}
Figure~\ref{fig:ce} displays the four ingredients of the mechanism. 
Writing
$S(t):=\sup_{1\le x\le2}|\partial_x\log p_t(x)|$ and $W(t):=2\big(1-F_t(1)\big)$ for the
length of the emitting interval:

\begin{center}
\begin{tabular}{r|cc|cc|cc}
$t$ & $S(t)$ & $t\,S(t)$ & $\Fish(p_t)$ & $\sqrt t\,\Fish(p_t)$ & $W(t)$ & $\sqrt{t/2\pi}$\\
\hline
$0.5$   & $2.638$   & $1.319$ & $1.210$  & $0.8558$ & $0.281606$ & $0.282095$\\
$0.1$   & $10.860$  & $1.086$ & $2.856$  & $0.9032$ & $0.126157$ & $0.126157$\\
$0.01$  & $100.981$ & $1.010$ & $9.032$  & $0.9032$ & $0.039894$ & $0.039894$\\
$0.002$ & $500.996$ & $1.002$ & $20.196$ & $0.9032$ & $0.017841$ & $0.017841$\\
\end{tabular}
\end{center}

\noindent
The second column confirms $S(t)=t^{-1}(1+o(1))$, so $S\notin L^1(0,T)$ and (S4) fails; the
fourth confirms $\Fish(p_t)=c\,t^{-1/2}(1+o(1))$ with $c\approx0.9032$, so
$\Fish\in L^1(0,T)$ and (S4$'$) holds. The $t^{-1/2}$ rate is what one expects from the two
jump discontinuities of $p_0$: each contributes an initial layer of width $O(\sqrt t)$ on
which $|\partial_x\log p_t|=O(t^{-1/2})$. The last two columns confirm
$W(t)=\sqrt{t/2\pi}\,(1+o(1))$.

It is worthwhile noting that in this example
 the transport identity holds:
    by Theorem~\ref{thm:ce-fail}(iii) the map $Z(t,\cdot)$ is defined on $(-1,1)$, which carries
full $p_0$-measure, and by \eqref{eq:cejac} it satisfies
$Z(t,\cdot)_\#(p_0\Leb)=p_t\Leb$ for every $t\in[0,T]$. The conclusion of
Theorem~\ref{thm:lagrangian}\textup{(iii)} is \emph{true} here even though its assumptions fail. Moreover the compressibility bound
Definition~\ref{def:rlf}(ii) holds on $(-1,1)$, the pushforward density being
$1/|\partial_xZ|=2p_t\le1$.

Points that lie outside $[-1,1]$ at time $t$ are the images of the interval
$\big(2F_t(1)-1,\,1\big)$, whose length is $2\big(1-F_t(1)\big)\sim\sqrt{t/2\pi}$ --- small,
but of \emph{positive} measure for each $t>0$. There is no ``creation of mass from a single
point. Failure of well-posedness is therefore a statement about the Cauchy problem for
arbitrary data, and about the non-existence of a flow on the entire real line; it does not
contradict   the marginal-matching identity for a specific  $p_0$.

\begin{remark}[Generality of the construction]\label{rem:ce-variants}
Nothing in the example depends on the particular shape of $p_0$. Any compactly supported density produces
the same $1/t$ blow-up outside its support, however smooth, since the obstruction is that
$F_0$ attains its extreme values $0$ and $1$ on sets of positive measure. Taking
$p_0=\tfrac12\mathbf 1_{[-2,-1]}+\tfrac12\mathbf 1_{[1,2]}$ places a vacuum region in the
interior, with characteristics filling $(-1,1)$ from both edges. What is required to
restore Theorem~\ref{thm:lagrangian} is a globally positive $p_0$ with tails no lighter
than Gaussian --- which is what the two-sided condition
$-\beta_0I\preceq\nabla^2\log p_0\preceq-\alpha_0I$ of Proposition~\ref{prop:linear}
enforces.
\end{remark}

\section{Stability under a learned score function}\label{sec:stability}

In diffusion models $\nabla\log p_t(x)$ is replaced by a learned score function $s_\theta(x,t)$, and one integrates \cite{Song2021mle,hanxu2024}:
\begin{equation}\label{eq:pfodetheta}
  \tfrac{\dd}{\dd t}Z^\delta_\theta(t,x)=v_\theta\big(Z^\delta_\theta(t,x),t\big),\qquad
  v_\theta(x,t):=f(x,t)-\tfrac12\sigma^2(t) s_\theta(x,t),\qquad Z^\delta_\theta(\delta,x)=x .
\end{equation}

Both flows are started at time $\delta$, not at  $t=0$. All assumptions in this
section are imposed on $[\delta,T]$ only, for the reason set out in
Remark~\ref{rem:corfix}: for data of interest \textup{(S)} simply fails on $(0,\delta)$.
Accordingly we write, for $\delta\in(0,T)$,
\[
  Z^\delta(t,\cdot),\quad Z^\delta_\theta(t,\cdot),\qquad t\in[\delta,T],
\]
for the regular Lagrangian flows of $v$ and of $v_\theta$ on $[\delta,T]$ \emph{normalised
by} $Z^\delta(\delta,x)=Z^\delta_\theta(\delta,x)=x$. The forward flow of
Theorem~\ref{thm:lagrangian} factors as $Z(t,\cdot)=Z^\delta(t,\cdot)\circ Z(\delta,\cdot)$
for $t\ge\delta$, $\Leb$-a.e., and $Z^\delta(t,\cdot)_\#\big(p_\delta\Leb\big)=p_t\Leb$ by
Theorem~\ref{thm:lagrangian}\textup{(iii)} applied on $[\delta,T]$ with initial condition  $p_\delta$.

Theorem~\ref{thm:lagrangian} applied to $v_\theta$ says \eqref{eq:pfodetheta} is well posed;
it says nothing about $\|Z^\delta-Z^\delta_\theta\|$. 
In this section we provide error estimates on the flow using the
 a priori estimates of Crippa--De~Lellis \cite{CDL2008}.
As shown in Theorem \ref{thm:cdl}, this error is controlled by the $L^1$ error on the velocity $\mathcal{E}_v:=\big\|v-v_\theta\big\|_{L^1([\delta,T]\times B_R)},$
 while the score is trained using a probability-weighted $L^2$ norm \eqref{eq.weightedloss} with weight $p_t(x)$. We discuss this discrepancy and its consequences in Section \ref{sec.conversion}.
 
\subsection{Assumptions on the learned score function}
The score function is learned using a parametric model or neural network $s_\theta(x,t)$, for which we make the following assumptions.
\begin{assumption}[Learned score --- (A)]\label{ass:A}
For some $p>1$ and some $\delta\in[0,T)$, all conditions being imposed on $[\delta,T]$.
\begin{enumerate}
\item[(A1)] $s_\theta\in L^1\big([\delta,T];W^{1,p}_\loc(\R^d;\R^d)\big)$;
\item[(A2)] $[\Div v_\theta]^-\in L^1\big([\delta,T];L^\infty\big)$.
\item[(A3)] $|s_\theta(x,t)|\le A_\theta(s)(1+|x|)$ with $A_\theta\in L^1(\delta,T)$.
\end{enumerate}
We also strengthen \textup{(S2)} to:\ 
[(S2$_p$)] $\nabla^2\log p_\cdot\in L^1\big([\delta,T];L^p_\loc\big)$ for the same $p>1$.\\
We \emph{replace} the drift assumptions \textup{(D1)} and \textup{(D3)}
by  stronger  assumptions   on $[\delta,T]$:
\begin{enumerate}
\item[(D1$_p$)] $f\in L^1\big([\delta,T];W^{1,p}_\loc(\R^d;\R^d)\big)$ for the same $p>1$;
\item[(D4)] $|f(x,t)|\le C_f(t)\,(1+|x|)$ for a.e.\ $(x,t)\in \mathbb{R}^d\times [\delta, T]$, with $C_f\in L^1(\delta ,T)$.
\end{enumerate}
\end{assumption}
 Here the  growth bound (D4)  of Assumption~\ref{ass:D} is only required for
$t\in [\delta,T]$. The case $\delta=0$ is permitted but is available only when \textup{(S)} holds up to $t=0$,
which by Remark~\ref{rem:corfix} requires the two-sided condition \eqref{eq:p0hess}.
We write \textup{(D$_{\mathrm S}$)} for \textup{(D0)}, \textup{(D2)}, \textup{(D1$_p$)},
\textup{(D4)}.

Note that these are assumptions on the \emph{architecture}, not the loss function or the training algorithm.
The
score-matching objective constrains $s_\theta$ in $L^2(p_s\dd x\dd s)$; it says nothing
about $\Div s_\theta$. By Theorem \eqref{thm:DL},
(A1)--(A3) ensure that
\eqref{eq:pfodetheta} has a regular Lagrangian flow $Z^\delta_\theta$ with compression constant
$L_\theta=\exp\|[\Div v_\theta]^-\|_{L^1(L^\infty)}$.
Compressibility of the learned flow is therefore an
assumption on the network that must be imposed or verified separately e.g.\ by a
Jacobian-trace constraint. 

\subsection{A stability estimate for the ODE flow}

\begin{theorem}[Log-Lipschitz stability of the PF-ODE]\label{thm:cdl}
Assume \textup{(E1)}, \textup{(D$_{\mathrm S}$)}, \textup{(I)}, \textup{(S)} with \textup{(S2$_p$)}, and
\textup{(A)}. Fix $r>0$ and set
\[
  R:=(1+r)\,\exp\Big(\|G\|_{L^1}+\|G_\theta\|_{L^1}\Big),\qquad
  \mathcal{E}_v:=\big\|v-v_\theta\big\|_{L^1([\delta,T]\times B_R)},
\]
where $G,G_\theta\in L^1$ are the linear-growth rates of $v,v_\theta$ from
Theorem~\ref{thm:lagrangian}\textup{(i)} and \textup{(A3)}. Then for every $\lambda>0$ and
$\eta>0$,
\begin{equation}\label{eq:cdlmeas}
  \Leb\Big(\big\{x\in B_r:\ \big|Z^\delta(T,x)-Z^\delta_\theta(T,x)\big|>\eta\big\}\Big)
  \;\le\;\frac{1}{\log\!\big(1+\eta/\lambda\big)}
  \Big[\,\mathcal C_p\;+\;\frac{L_\theta}{\lambda}\,\mathcal E_v\,\Big],
\end{equation}
where
\[
  \mathcal C_p:=C_{d,p}\,(L+L_\theta)\,|B_R|^{1/p'}\int_\delta^T\!\big\|\nabla v(\cdot,s)\big\|_{L^p(B_{3R})}\dd s,
  \qquad \tfrac1{p'}=1-\tfrac1p .
\]
Optimising with $\lambda=\sqrt{\mathcal E_v}$ gives, for $\mathcal E_v<\eta^2$,
\begin{equation}\label{eq:cdlrate}
  \Leb\Big(\big\{x\in B_r:\ |Z^\delta(T,x)-Z^\delta_\theta(T,x)|>\eta\big\}\Big)
  \;\le\;\frac{\mathcal C_p+L_\theta\sqrt{\mathcal E_v}}{\log\!\big(1+\eta\,\mathcal E_v^{-1/2}\big)}
  \;\xrightarrow[\ \mathcal E_v\to0\ ]{}\;0 .
\end{equation}
The convergence is \emph{logarithmic} in $\mathcal E_v$.
\end{theorem}

\begin{proof}
Throughout this proof we abbreviate $Z:=Z^\delta$ and $Z_\theta:=Z^\delta_\theta$. Both are
normalised at time $\delta$ by $Z(\delta,x)=Z_\theta(\delta,x)=x$; this is exactly what makes
$\Phi_\lambda(\delta)=0$ legitimate in Step 2, and it is false for the flows started at time
$0$.

$e(t,x):=|Z(t,x)-Z_\theta(t,x)|$ is absolutely continuous in $t$ with
$\dot{e}\le|v(Z,t)-v_\theta(Z_\theta,t)|$. Split
\[
  \dot{e}\;\le\;\underbrace{\big|v(Z,t)-v(Z_\theta,t)\big|}_{\text{regularity of }v}
  \;+\;\underbrace{\big|(v-v_\theta)(Z_\theta,t)\big|}_{\text{model error}} .
\]
$Z^\delta$ and $Z^\delta_\theta$ are the flows normalised at time $\delta$, so
that $e(\delta,\cdot)\equiv 0$ and $\Phi_\lambda(\delta)=0$ below.

\emph{Step 1: confinement.} Here \textup{(D4)} is used, and \textup{(D3)} would not
suffice: combining \textup{(D4)} with \textup{(S4)} and \textup{(A3)} gives the
 pointwise bounds
\[
  |v(x,s)|\le G(s)(1+|x|),\qquad |v_\theta(x,s)|\le G_\theta(s)(1+|x|),
\]
with $G:=C_f+\tfrac12\epsilon^{-2}A$ and $G_\theta:=C_f+\tfrac12\epsilon^{-2}A_\theta$, both
in $L^1(\delta,T)$. Gr\"onwall then gives
$|Z^\delta(s,x)|,|Z^\delta_\theta(s,x)|\le(1+|x|)e^{\|G\|_{L^1}+\|G_\theta\|_{L^1}}$. Hence both
trajectories started in $B_r$ remain in $B_R$ for all $s\in[\delta,T]$.

\emph{Step 2: the logarithmic functional.} For $\lambda>0$ set
\[
  \Phi_\lambda(s):=\int_{B_r}\log\Big(1+\frac{e(s,x)}{\lambda}\Big)\dd x,
  \qquad \Phi_\lambda(\delta)=0 .
\]
Then
\[
  \Phi_\lambda'(s)\;\le\;\int_{B_r}\frac{\big|v(Z,s)-v(Z_\theta,s)\big|}{\lambda+ e(s,x)}\dd x
  \;+\;\frac1\lambda\int_{B_r}\big|(v-v_\theta)(Z_\theta,s)\big|\dd x
  \;=:\;\mathrm{I}(s)+\mathrm{II}(s).
\]

\emph{Step 3: term $\mathrm{I}$ via maximal functions.} This is where
\textup{(D1$_p$)} is used. By \textup{(D1$_p$)} and \textup{(S2$_p$)},
\[
  \nabla v=\nabla f-\tfrac12\sigma^2\nabla^2\log p\ \in\ L^1\big([\delta,T];L^p_\loc\big),
  \qquad\text{so}\qquad v(\cdot,s)\in W^{1,p}_\loc\ \text{ with }p>1 .
\]
It is essential that \emph{both} summands be Sobolev: under \textup{(D1)} alone the singular
part $D^sf$ would survive into $Dv$ and $v$ would only be $\BV_\loc$.

For $\ell>0$ let $M_\ell g(x):=\sup_{0<u<\ell}\int_{B_u(x)}|g|$ denote the
\emph{truncated} maximal operator at scale $\ell>0$. For $u\in W^{1,p}_\loc$ we have
\begin{equation}\label{eq:lusin}
  |u(x)-u(y)|\le C_d\,|x-y|\Big(M_\ell|\nabla u|(x)+M_\ell|\nabla u|(y)\Big)
  \qquad\text{for a.e.\ }x,y\text{ with }|x-y|\le\ell ,
\end{equation}
together with the local bound $\|M_\ell g\|_{L^p(B_R)}\le C_{d,p}\|g\|_{L^p(B_{R+\ell})}$
for $p>1$. By Step 1 both $Z^\delta(s,x)$ and $Z^\delta_\theta(s,x)$ lie in $B_R$, so their
separation never exceeds $2R$; we may therefore take $\ell=2R$, and the enlarged ball in
\eqref{eq:lusin} is $B_{3R}$. 
Applying \eqref{eq:lusin} and using $e/(\lambda+e)\le1$ 
\[
  \mathrm{I}(s)\le C_d\int_{B_r}\Big(M_{2R}|\nabla v|(Z^\delta(s,x))+M_{2R}|\nabla v|(Z^\delta_\theta(s,x))\Big)\dd x ,
\]
where $\lambda$ was introduced in Step 2.
Pushing forward by the two flows, using their compressibility constants $L,L_\theta$ and
Step 1, and the H\"older inequality
\[
  \mathrm{I}(s)\le C_d\,(L+L_\theta)\int_{B_R}M_{2R}|\nabla v(\cdot,s)|\,\dd y
  \le C_d\,(L+L_\theta)\,|B_R|^{1/p'}\big\|M_{2R}|\nabla v(\cdot,s)|\big\|_{L^p(B_R)} .
\]
The maximal operator is bounded on $L^p$ \emph{for $p>1$}, with constant $C_{d,p}$, whence
$\|M_{2R}|\nabla v|\|_{L^p(B_R)}\le C_{d,p}\|\nabla v\|_{L^p(B_{3R})}$ and
$\int_\delta^T\mathrm{I}(s)\dd s\le\mathcal C_p$. The maximal operator maps $L^1$  into weak-$L^1$: the failure of the $L^p$ bound at $p=1$ 
is the reason why
Crippa--De~Lellis \cite{CDL2008} require $p>1$.

For $\mathrm{II}$, pushing forward by $Z_\theta$ and using Step 1,
\[
  \int_\delta^T\mathrm{II}(s)\dd s\le\frac{L_\theta}{\lambda}
  \int_\delta^T\!\!\int_{B_R}\big|v-v_\theta\big|\,\dd y\,\dd s=\frac{L_\theta}{\lambda}\,\mathcal E_v .
\]
Combining the above we obtain $\Phi_\lambda(T)\le\mathcal C_p+L_\theta\mathcal E_v/\lambda$.
On the set where $e(T,x)>\eta$ the integrand of $\Phi_\lambda(T)$ exceeds
$\log(1+\eta/\lambda)$, giving \eqref{eq:cdlmeas}; \eqref{eq:cdlrate} is obtained by choosing
$\lambda=\sqrt{\mathcal E_v}$.
\end{proof}

\begin{corollary}[$L^1$ and Wasserstein error estimates]\label{cor:wass}
Let
\[
  D_r:=2(1+r)\,e^{\|G\|_{L^1}+\|G_\theta\|_{L^1}}
\]
be the trajectory-diameter bound of Step 1. Then for every $\eta>0$,
\begin{equation}\label{eq:wassL1}
  \big\|Z^\delta(T,\cdot)-Z^\delta_\theta(T,\cdot)\big\|_{L^1(B_r)}
  \le\eta\,|B_r|+D_r\cdot\frac{\mathcal C_p+L_\theta\sqrt{\mathcal E_v}}{\log(1+\eta\,\mathcal E_v^{-1/2})}
  \;=:\;\mathcal R(\eta,r).
\end{equation}
The  initial law is $p_\delta$, \emph{not} $p_0$, since both flows are normalised at
time $\delta$. As $p_\delta$ will in general have full support,
the tail must be estimated separately. Using Step 1 to bound
$|Z^\delta(T,x)-Z^\delta_\theta(T,x)|\le 2e^{\|G\|_{L^1}+\|G_\theta\|_{L^1}}(1+|x|)$ off
$B_r$, one gets for every $r,\eta>0$
\begin{equation}\label{eq:wass}
  W_1\Big(Z^\delta(T,\cdot)_\#p_\delta,\ Z^\delta_\theta(T,\cdot)_\#p_\delta\Big)
  \;\le\;\|p_\delta\|_{L^\infty}\,\mathcal R(\eta,r)
  \;+\;2e^{\|G\|_{L^1}+\|G_\theta\|_{L^1}}\!\!\int_{|x|>r}\!\!(1+|x|)\,p_\delta(x)\,\dd x .
\end{equation}
The tail integral tends to $0$ as $r\to\infty$ because $p_\delta$ has a finite first moment
by \textup{(S1)}, so for \emph{fixed} $\theta$-dependent constants $L_\theta$, $G_\theta$ and
$\mathcal C_p$, optimising over $r$ and $\eta$ makes the right-hand side of \eqref{eq:wass}
tend to $0$ as $\mathcal E_v\to 0$. Since
$Z^\delta(T,\cdot)_\#p_\delta=p_T$ by Theorem~\ref{thm:lagrangian}\textup{(iii)} applied on
$[\delta,T]$, the left-hand side is the $W_1$ distance between $p_T$ and the output of the
deterministic sampler run from $p_\delta$ with the learned score.
\end{corollary}

\begin{remark}[Convergence along a sequence of learned score functions]
\label{rem:uniformseq}
Theorem~\ref{thm:cdl} and Corollary~\ref{cor:wass} are estimates for a \emph{fixed} $\theta$:
every constant on the right-hand side ( $L_\theta$, $G_\theta$, $D_r$, $\mathcal C_p$, and
the radius $R$ itself) is allowed to depend on $s_\theta$. The assertion that the bound
vanishes as $\mathcal E_v\to 0$ must therefore be understood with those constants held fixed. It
does \emph{not}  yield convergence along a sequence $(s_{\theta_n})$ with
$\mathcal E_v^{(n)}\to0$, and such a statement requires uniform bounds:
\begin{enumerate}
\item[\textup{(U1)}] $\displaystyle\sup_n L_{\theta_n}
  =\sup_n\exp\big\|[\Div v_{\theta_n}]^-\big\|_{L^1(0,T;L^\infty)}<\infty$;
\item[\textup{(U2)}] $\displaystyle\sup_n\|A_{\theta_n}\|_{L^1(\delta,T)}<\infty$ in
  \textup{(A3)}, equivalently $\sup_n\|G_{\theta_n}\|_{L^1}<\infty$;
\item[\textup{(U3)}] $\mathcal C_p<\infty$ for the radius $R$ resulting from \textup{(U2)},
  i.e.\ $\nabla v\in L^1\big([\delta,T];L^p(B_{3R})\big)$.
\end{enumerate}
\textup{(U2)} is not a technicality. By Step 1 the confinement radius is
$R=(1+r)e^{\|G\|_{L^1}+\|G_{\theta}\|_{L^1}}$, so if $\|G_{\theta_n}\|_{L^1}$ is unbounded
then $R_n\to\infty$; and then the domain $[\delta,T]\times B_{R_n}$ on which
$\mathcal E_v^{(n)}$ is measured varies with $n$, as does
$\mathcal C_p=\mathcal C_p(R_n)$ through $\|\nabla v\|_{L^p(B_{3R_n})}$ and through the factor
$|B_{R_n}|^{1/p'}$.  Under \textup{(U1)--(U3)} the radius may be fixed once and for all,
the constants are uniform, and \eqref{eq:cdlrate} gives
\[
  \Leb\Big(\big\{x\in B_r:\ \big|Z^\delta(T,x)-Z^\delta_{\theta_n}(T,x)\big|>\eta\big\}\Big)
  \;\xrightarrow[n\to\infty]{}\;0\qquad\text{for every }\eta>0 .
\]
\textup{(U1)--(U3)} are not determined by the training objective, which controls only
$\mathcal E_{\theta}$.
\end{remark}

\subsection{From the score-matching loss to $\mathcal E_v$}\label{sec.conversion}
The stability estimate in Theorem \ref{thm:cdl} is expressed in terms of an
\emph{unweighted} $L^1$ norm $\mathcal E_v$
, whereas score matching  is based on a probability-weighted $L^2$ norm with weight $p_t(x)$.
We now explore  the relation between these two quantities. Proposition \ref{prop:transfer} relates the training cost to the norm required by the flow estimate, thus relating score approximation to sampler stability, and identifying  the mismatch between the training and stability norms in low density regions.

\begin{proposition}[Relation between velocity error and score matching loss]\label{prop:transfer}
Let
\begin{equation}
    \mathcal {E}_\theta^2:=\int_\delta^T\!\!\int_{\R^d}p_t(x)\,\big|s_\theta(x,t)-\nabla\log p_t(x)\big|^2\,\dd x\,\dd t \label{eq.weightedloss}
\end{equation}
be the  denoising score-matching loss function. If the conversion factor
\[
\kappa_{\delta,R}:=\int_\delta^T\!\!\int_{B_R}\frac{\dd x\,\dd s}{p_s(x)} <\infty
\]
\[{\rm is\  finite, \ then}\qquad
  \mathcal E_v\;=\;\tfrac12\big\|\sigma^2\big(\nabla\log p-s_\theta\big)\big\|_{L^1([\delta,T]\times B_R)}
  \;\le\;\frac{\|\sigma\|_{L^\infty}^2}{2}\mathcal{E}_\theta\;\sqrt{\kappa_{\delta,R}}.
\]
\end{proposition}

\begin{proof}
Write $|\nabla\log p_s-s_\theta|=\big(\sqrt{p_s}\,|\nabla\log p_s-s_\theta|\big)\cdot p_s^{-1/2}$
and apply Cauchy--Schwarz on $[\delta,T]\times B_R$; then bound $\sigma^2\le\|\sigma\|^2_{L^\infty}\leq \epsilon^{-2}$ by (E1).
\end{proof}
Finiteness of $\kappa_{\delta,R}$ is an \emph{additional assumption}; it requires a lower bound on the density and does not follow from
\textup{(E1)--(E2)} and \textup{(D$_{\mathrm E}$)}. The following lemma shows that it is satisfied in the case of linear diffusion models.
\begin{lemma}[$\kappa_{\delta,R}<\infty$ for linear drift]\label{lem:kappalinear}
Assume \textup{(E1)} and $f(x,s)=-\beta(s)x$ with $\beta\in L^\infty(0,T)$ --- the
variance-preserving and variance-exploding models of Corollary~\ref{cor:diffusion}. Let $p_0$
be any probability density. With $m_s:=e^{-\int_0^s\beta}$ and
$\varsigma_s^2:=m_s^2\int_0^s\sigma(r)^2m_r^{-2}\dd r$, set
$\varsigma_{\min}^2:=\inf_{s\in[\delta,T]}\varsigma_s^2>0$,
$\varsigma_{\max}^2:=\sup_{s\in[\delta,T]}\varsigma_s^2<\infty$ and
$m_{\max}:=\sup_{s\le T}m_s$. Fix any $\rho>0$ with $\mu_\rho:=\int_{B_\rho}p_0\,\dd x>0$.
Then for all $s\in[\delta,T]$ and $|x|\le R$,
\begin{equation}\label{eq:kappalower}
  p_s(x)\;\ge\;\frac{\mu_\rho}{(2\pi\varsigma_{\max}^2)^{d/2}}
  \exp\!\Big(-\frac{(R+m_{\max}\rho)^2}{2\varsigma_{\min}^2}\Big)\;>\;0 ,
\end{equation}
and consequently
\[
  \kappa_{\delta,R}\;\le\;(T-\delta)\,|B_R|\,
  \frac{(2\pi\varsigma_{\max}^2)^{d/2}}{\mu_\rho}\,
  \exp\!\Big(\frac{(R+m_{\max}\rho)^2}{2\varsigma_{\min}^2}\Big)\;<\;\infty .
\]
\end{lemma}
The proof, given in Appendix \ref{app.linear}, uses Gaussian convolution.

We now give a sufficient condition for $\kappa_{\delta,R}<\infty$ for a general class of nonlinear velocity fields. The proof is given in Appendix \ref{app.nonlinear}.
\begin{lemma}[$\kappa_{\delta,R}<\infty$ for nonlinear drift]
\label{lem:kappabdd}
Let $\sigma$ satisfy
\textup{(E1)--(E2)} and  $f\in L^\infty\big([0,T]\times\R^d;\R^d\big)$ 
with \emph{no} condition  on $\Div f$, which may be a measure with a nonzero singular
part. Let $p_0$ be a probability density with $\mu_\rho:=\int_{B_\rho}p_0\,\dd x>0$ for some
$\rho>0$, and let $p_s$ denote the density of the law of $X_s$, the solution of
\eqref{eq:sde}, which exists and is unique in law under \textup{(E1)}.
Then $\kappa_{\delta,R}<\infty$ for every $0<\delta<T$ and $R>0$.
\end{lemma}
Note that we do not assume \eqref{eq:divac}. We use Aronson's estimates, which require only that the divergence-form operator have bounded
measurable coefficients, a condition on $f$ itself and not on its derivative. Thus an example such as
$f(x,t)=-\operatorname{sign}(x)$ in Remark~\ref{rem:divsing}, which is excluded from
Theorem~\ref{thm:eulerian} precisely because $\Div f=-2\delta_0$ is singular, is covered
here.

Lemmas~\ref{lem:kappalinear} and~\ref{lem:kappabdd} cover disjoint situations.
Bounded $f$ excludes the variance-preserving model, whose drift $-\beta(s)x$ is unbounded;
Lemma~\ref{lem:kappalinear} covers it, but only because the linear structure makes $p_s$ an
explicit Gaussian convolution and no heat-kernel theory is needed. Conversely
Lemma~\ref{lem:kappabdd} allows drifts far outside the linear class --- including the
$f=-\operatorname{sign}$ field of Remark~\ref{rem:divsing}, whose divergence is a measure ---
at the price of global boundedness.

In the intermediate case of nonlinear drifts with linear growth one may obtain similar statements for $f\in L^q\big([0,T];L^r(\R^d)\big)$ with
$\tfrac dr+\tfrac2q<1$; see \cite{CKP2012}. 

\begin{remark}\label{rem:kappa}
In these examples, $\kappa_{\delta,R}$ is finite but its size is governed by $\inf_{[\delta,T]\times B_R}p_s$,
which is exponentially small in $R^2/\varsigma_\delta^2$ by the Gaussian lower bound and
degenerates as $\delta\downarrow0$. So the score can be learned to high accuracy in high density regions yet 
still have low accuracy on the low-density regions.  This discrepancy is detected by the unweighted norm, which is
precisely the norm entering the flow-stability estimate. The chain
\[
  \underbrace{\mathcal{E}_\theta}_{\text{trainable}}
  \;\xrightarrow[\ \times\sqrt{\kappa_{\delta,R}}\ ]{\text{Prop.\ \ref{prop:transfer}}}\;
  \underbrace{\mathcal E_v}_{L^1}
  \;\xrightarrow[\ \text{logarithmic}\ ]{\text{Thm.\ \ref{thm:cdl}}}\;
  \underbrace{\|Z^\delta-Z^\delta_\theta\|}_{\text{error on }[\delta,T]}
\]
\end{remark}

\subsection{ Lipschitz score functions}

When the score function is approximated by a neural network, one can easily enforce  Lipschitz continuity by using smooth activation functions and 'clipping' (bounding) the weights \cite{hanxu2024}.
In this case,
one can bypass the above estimates entirely for a linear,   rather than logarithmic, rate.
\begin{proposition}[Gr\"onwall stability]\label{prop:gronwall}
Assume the \emph{time-dependent} bounds
\[
  \opnorm{\nabla^2\log p_s}\le M(s),\qquad \|\nabla f(\cdot,s)\|_\infty\le F(s),
  \qquad F,\ \sigma^2M\in L^1(\delta,T).
\]
Then $v(\cdot,s)$ is Lipschitz with constant
$\mathrm{Lip}_s(v)\le F(s)+\tfrac12\sigma(s)^2M(s)$, and for every $x$
\begin{equation}\label{eq:gronrate}
  \sup_{s\in[\delta,T]}\big|Z^\delta(s,x)-Z^\delta_\theta(s,x)\big|
  \;\le\;\underbrace{\exp\!\Big(\!\int_\delta^T\!\!\big(F(s)+\tfrac12\sigma(s)^2M(s)\big)\dd s\Big)}_{=:\,\mathcal A_\delta}
  \cdot
  \tfrac12\int_\delta^T\!\sigma(s)^2\big\|\nabla\log p_s-s_\theta(\cdot,s)\big\|_{L^\infty}\dd s .
\end{equation}
\end{proposition}
\begin{proof}
$\dot{e}\le\mathrm{Lip}_s(v)\,e+\|(v-v_\theta)(\cdot,s)\|_\infty$ pointwise. Applying
Gr\"onwall with $e(\delta,x)=0$, which again uses the normalisation
$Z^\delta(\delta,x)=Z^\delta_\theta(\delta,x)=x$ yields the estimate.
\end{proof}

\begin{remark}\label{rem:trade}
Proposition~\ref{prop:gronwall} leads to a linear error estimate, far better than
\eqref{eq:cdlrate}, but it pays for this in the amplification factor $\mathcal A_\delta$ of
\eqref{eq:gronrate}. The size of $\mathcal A_\delta$ depends  on whether $M$ is
taken uniform or time-dependent.

\emph{With a uniform bound.} Proposition~\ref{prop:earlystop} gives, for early-stopped
compactly supported data, $\sup_{s\ge\delta}M(s)=O(R^2\varsigma_{\min}^{-4})$ --- \emph{not}
$O(\varsigma_{\min}^{-2})$. Inserting this constant into
\eqref{eq:gronrate} yields only
\[
  \mathcal A_\delta\;\lesssim\;\exp\!\big(C R^2\varsigma_{\min}^{-4}\big)
  \;=\;\exp\!\big(CR^2\delta^{-2}\big)
  \qquad\text{for }\varsigma_s^2=s .
\]

\emph{With the time-dependent bound.} Keeping $M(s)$ inside the integral and using the
pointwise form of \eqref{eq:esS2}, $M(s)\le\varsigma_s^{-2}+m_{\max}^2R^2\varsigma_s^{-4}$,
the exponent becomes, for the variance-exploding model with $\sigma\equiv\sigma_0$ and
$\varsigma_s^2=\sigma_0^2s$,
\[
  \tfrac12\!\int_\delta^T\!\sigma_0^2M(s)\,\dd s
  \;\le\;\tfrac12\log\frac T\delta+\frac{R^2}{2\sigma_0^{2}}\Big(\frac1\delta-\frac1T\Big),
  \qquad\text{since }\int_\delta^T\! s^{-2}\dd s=\delta^{-1}-T^{-1},
\]
so that $\mathcal A_\delta\lesssim(T/\delta)^{1/2}\exp\!\big(CR^2\delta^{-1}\big)$. The
$\delta^{-2}$ in the exponent is therefore an artefact of the uniform bound: the true rate is
$\exp(CR^2/\delta)$, because $\varsigma_s^{-4}$ is integrable in $s$ away from the origin
while its supremum is not.

Either way the factor blows up as the sampler approaches the data end, which is the analytic
reason why the final reverse steps are the delicate ones. 

The two regimes are complementary and neither is fully satisfactory:
Crippa--De~Lellis \cite{CDL2008} gives an $L^1$ spatial error and mild regularity at the price of a
logarithmic rate, while the Gr\"onwall estimate \eqref{eq:gronrate} gives a linear rate with an $L^\infty$ error and an
exponentially large constant. Neither delivers a bound that is quantitatively useful at
realistic values of $\mathcal{E}_\theta$. The logarithmic rate is not an artefact of the
proof: it is known to be essentially optimal for   Sobolev velocity fields \cite{lellis2008}.
\end{remark}

\begin{remark}[Error analysis for generative models]\label{rem:missing}
Our error estimate focuses on the impact of the estimation error for the score $s$. There are other sources of error which affect the accuracy of score-based generative models \cite{hanxu2024}. 

\emph{Early stopping:} Theorem~\ref{thm:cdl} compares two flows
\emph{normalised at time $\delta$}, and therefore bounds only the error generated on
$[\delta,T]$. Writing $Z(T,\cdot)=Z^\delta(T,\cdot)\circ Z(\delta,\cdot)$ and similarly for
$Z_\theta$, the missing piece is the propagation through $Z^\delta_\theta(T,\cdot)$ of the
discrepancy $|Z(\delta,x)-Z_\theta(\delta,x)|$ already accumulated on $[0,\delta]$.

\emph{Time discretisation:} everything above compares exact flows of $v$ and
$v_\theta$, whereas a sampler integrates \eqref{eq:pfodetheta} numerically, and the
integrator error must be composed with \eqref{eq:cdlrate}.

\emph{Truncation error:} $\mathcal{E}_v$ is measured on $B_R$ with $R$ growing exponentially in
$\|G\|_{L^1}$, so the tail contribution is not controlled by a risk computed under $p_s$.

\emph{Finite horizon:} the initial condition of the reverse sampler is $\mathcal N(0,\varsigma_T^2I)$, which is the invariant distribution,
rather than $p_T$. That mismatch is a separate additive term, controlled by the mixing of the
forward process and \emph{not} by anything above.

Han et al.\cite{hanxu2024} provide estimates for some of these error terms in the case where the score is estimated using a neural network with one hidden layer.
\end{remark}

\section{Implications for generative modeling}\label{sec.discussion}
We  conclude by outlining some implications of our results for generative modeling with diffusion models. 
\begin{enumerate}
\item {\bf A well-posed forward diffusion does not automatically yield a well-posed probability-flow ODE}
The Fokker--Planck density flow may be uniquely determined under relatively weak assumptions
on the drift and diffusion coefficient, without substantial regularity of the score. By contrast,
the probability-flow velocity field
\[
v(x,t)
=
f(x,t)-\frac{1}{2}\sigma(t)^2\nabla\log p_t(x)
\]
inherits the singularities and irregularities of the score.
Consequently, existence and uniqueness of the deterministic flow require a separate analysis.
In particular, one must verify (Sobolev or \(BV\)) regularity, divergence bounds,
and growth conditions for \(v\). The statement that the probability-flow ODE has the same
one-time marginals as the SDE is therefore contingent on various assumptions, rather than a purely
formal consequence of rewriting the Fokker--Planck equation.
\item{\bf  Smoothing does not remove the singularity of the data endpoint}
For compactly supported or lower-dimensional data, the forward diffusion may produce a
strictly positive and smooth density \(p_t\) for  \(t>0\), while the score
$ \nabla\log p_t $
becomes singular as \(t\downarrow 0\).
Thus, smoothness at every  $t>0$ does not imply uniform regularity up to the data
endpoint. The marginal density flow may remain well posed even though a global regular
Lagrangian flow generated by the probability-flow velocity fails to exist on an interval
starting at \(t=0\).
This shows why the  reverse-time steps closest
to the data distribution are often the most delicate.
\item {\bf Early stopping is a regularization mechanism, not just a numerical device.}
For every fixed \(\delta>0\), the score estimates required by the regular Lagrangian-flow theory
may hold on the truncated interval
$
[\delta,T].$
One then proceeds as follows:
\[
\mu_0
\overset{\text{forward diffusion}}{\longrightarrow}
p_\delta\,\mathcal L^d
\overset{\text{probability-flow ODE}}{\longrightarrow}
p_t\,\mathcal L^d,
\qquad t\in[\delta,T].
\]

In the reverse direction, the deterministic probability-flow sampler is justified from
\(p_T\) down to \(p_\delta\). Reaching the original data law at \(t=0\), especially when that
law is singular, requires an additional limiting argument or a separate terminal denoising
step.
There is  a trade-off: decreasing \(\delta\) moves the sampler closer to the original
data distribution, but generally worsens the regularity constants entering the flow analysis.

\item {\bf Reverse sampling and invertibility require stronger hypotheses than forward transport}
A one-sided divergence bound may be sufficient to construct a forward regular Lagrangian
flow and to control its compression. Running the same flow backwards requires the complementary
one-sided estimate
\[
[\nabla\!\cdot v]^+
\in
L^1\!\left((0,T);L^\infty(\mathbb R^d)\right).
\]
Accordingly, the following assertions are stronger than forward marginal matching:
\begin{itemize}
    \item existence of a well-posed deterministic reverse sampler;
    \item almost-everywhere invertibility of the flow;
    \item deterministic encoding and decoding as inverse operations;
    \item likelihood evaluation through a change-of-variables formula.
\end{itemize}

These conclusions require two-sided control of the divergence. A forward probability-flow ODE may therefore be well posed even when the reverse
flow is not.
\item {\bf A small score-matching loss does not by itself imply a stable sampler}
A standard score-matching objective controls an error of the form
\[
\mathcal{E}_\theta^2
=
\int_\delta^T
\int_{\mathbb R^d}
p_t(x)
\left|
s_\theta(x,t)-\nabla\log p_t(x)
\right|^2
\,dx\,dt.
\]
This is a density-weighted error. By contrast, the flow-stability estimate generally requires an
unweighted spatial norm of the velocity error.
Passing from the weighted error to an unweighted local error introduces an inverse-density
quantity such as
\[
\kappa_{\delta,R}
=
\int_\delta^T
\int_{B_R}
\frac{dx\,dt}{p_t(x)}.
\]
Consequently, approximation errors in low-density regions may contribute very little to the
training loss while still affecting the global ODE vector field and its trajectories.

This reveals a mismatch between the metric used for training and the metric controlling
deterministic sampler stability. Score fitting alone may therefore be insufficient unless
it is combined with localisation, low-density control, or regularisation aligned with the
velocity-field error. This point differentiates the PF-ODE from the time-reversed diffusion, which is controlled by the score matching error.
\item {\bf Architectural regularity is part of sampler correctness}
The learned score enters the deterministic sampler through the velocity field
\[
v_\theta(x,t)
=
f(x,t)-\frac{1}{2}\sigma(t)^2s_\theta(x,t).
\]
A small score-matching loss controls the approximation \(s_\theta\) only in a
density-weighted norm, typically
$
L^2\!\bigl(p).$
It does not, by itself, control the spatial derivatives, divergence, growth, or
compressibility of the learned velocity field.
Accordingly the analysis requires
conditions such as
\[
s_\theta
\in
L^1\!\left(
[\delta,T];
W^{1,p}_{\mathrm{loc}}
(\mathbb{R}^d;\mathbb{R}^d)
\right),
\qquad p>1,
\]
on $[\delta,T]$ together with the one-sided divergence bound
\[
[\nabla\!\cdot v_\theta]^-
\in
L^1\!\left(
[\delta,T];
L^\infty(\mathbb{R}^d)
\right),
\]
and a linear-growth condition of the form
\[
|s_\theta(x,t)|
\le
A_\theta(t)(1+|x|),
\qquad
A_\theta\in L^1(\delta,T).
\]
These assumptions play different roles:
\begin{itemize}
    \item Sobolev regularity enables to   compare the exact and learned flows.
    \item The one-sided divergence bound controls compression and is essential
    for the existence, uniqueness, and stability of the corresponding regular
    Lagrangian flow.
    \item The linear-growth condition prevents trajectories from escaping to
    infinity during the interval on which the flows are compared.
\end{itemize}

The practical consequence is that a neural score model is not 
certified to yield a stable probability-flow ODE sampler merely because its
score-matching loss is small. {\it The architecture or regularisation procedure must
also impose, or one must separately verify, suitable Jacobian, divergence,
Sobolev, and growth bounds.}

In particular, a small score-matching loss
does not by itself imply an accurate simulation of the flow via the ODE.
For convergence of a sequence of learned samplers, the relevant regularity,
growth, and compressibility constants must remain uniformly controlled across
the sequence of models.
Spectral normalisation and Jacobian-based penalties may help enforce such
control. 
\end{enumerate}
Some of these remarks are valid more generally for other ODE-based continuous flows such as flow matching \cite{lipman2022}, which depend on the existence of corresponding Lagrangian flows.

%% file: appendices.tex
For $\eta\in(0,1]$, set
\[
\beta_\eta(r)
:=
(r+\eta)\log(r+\eta)-\eta\log\eta,
\qquad r\ge0.
\]
Then $\beta_\eta\in C^\infty([0,\infty))$ is convex and
\begin{equation}
\beta_\eta(0)=0,
\qquad
r\beta_\eta''(r)
=
\frac{r}{r+\eta}
\uparrow1
\quad\text{as }\eta\downarrow0,
\label{eq.15}
\end{equation}
while
\[
B_\eta(r)
:=
r\beta_\eta'(r)-\beta_\eta(r)
=
r-\eta\log\!\left(1+\frac r\eta\right).
\]
Since $\log(1+u)\le u$ for $u\ge0$, we have
$0\le B_\eta(r)\le r.$
We  use the one-sided estimates
\begin{equation}
r\log r
\le
\beta_\eta(r)
\le
r\log(1+r)+r.
\label{eq.16}
\end{equation}
Fix $\varepsilon>0$ and $\eta\in(0,1]$, and let $\psi_R$ be the
logarithmic cutoff defined in \eqref{eq:logcut}. Since $p^\varepsilon$ is a
classical solution and
\[
0\le p^\varepsilon
\le
\|p_0\|_{L^\infty}e^\theta,
\]
the function $\beta_\eta'(p^\varepsilon)$ is bounded for fixed
$\eta$. Multiplying the regularised equation by
$\beta_\eta'(p^\varepsilon)\psi_R$ and integrating over
$\mathbb R^d$ gives
\begin{align*}
\frac{d}{ds}
\int_{\mathbb R^d}
\beta_\eta(p^\varepsilon)\psi_R\,dx
&=
-\frac{\sigma(s)^2}{2}
\int_{\mathbb R^d}
\beta_\eta''(p^\varepsilon)
|\nabla p^\varepsilon|^2\psi_R\,dx
\\
&\quad
-\int_{\mathbb R^d}
B_\eta(p^\varepsilon)
(\nabla\!\cdot f^\varepsilon)\psi_R\,dx
+
\mathcal R_{\eta,R}^\varepsilon(s),
\end{align*}
where
\[
\mathcal R_{\eta,R}^\varepsilon(s)
:=
\int_{\mathbb R^d}
\beta_\eta(p^\varepsilon)
f^\varepsilon\cdot\nabla\psi_R\,dx
+
\frac{\sigma(s)^2}{2}
\int_{\mathbb R^d}
\beta_\eta(p^\varepsilon)\Delta\psi_R\,dx.
\]

The divergence term is estimated only from above. Since
$B_\eta\ge0$, $B_\eta(r)\le r$, and $0\le\psi_R\le1$,
\begin{align}
-\int_{\mathbb R^d}
B_\eta(p^\varepsilon)
(\nabla\!\cdot f^\varepsilon)\psi_R\,dx
&\le
\bigl\|[\nabla\!\cdot f^\varepsilon(\cdot,s)]^-\bigr\|_{L^\infty}
\int_{\mathbb R^d}B_\eta(p^\varepsilon)\,dx
\nonumber\\
&\le
\bigl\|[\nabla\!\cdot f(\cdot,s)]^-\bigr\|_{L^\infty}.
\label{eq.17}
\end{align}
Here we used mass conservation and the mollification estimate from
Step~$7(a)$,
\[
\bigl\|[\nabla\!\cdot f^\varepsilon]^-\bigr\|_{L^\infty}
\le
\bigl\|[\nabla\!\cdot f]^-\bigr\|_{L^\infty}.
\]
No global integrability of
$(\nabla\!\cdot f^\varepsilon)^+p^\varepsilon$ is required or
asserted.

For fixed $\eta$, the facts that $\beta_\eta(0)=0$ and
$p^\varepsilon$ is uniformly bounded imply
\[
|\beta_\eta(r)|
\le
c_\eta r,
\qquad
0\le r\le
\|p_0\|_{L^\infty}e^\theta,
\]
where
\[
c_\eta
:=
\sup_{0\le r\le\|p_0\|_{L^\infty}e^\theta}
|\beta_\eta'(r)|<\infty.
\]
Consequently, the two terms in
$\mathcal R_{\eta,R}^\varepsilon$ are bounded by $c_\eta$ times the
corresponding quantities estimated in Step~$7(c)$. The defining
properties of the logarithmic cutoff therefore give
\[
\int_0^T
\bigl|\mathcal R_{\eta,R}^\varepsilon(s)\bigr|\,ds
\longrightarrow0
\qquad\text{as }R\to\infty
\]
for each fixed $\eta$ and $\varepsilon$.

Integrating the renormalised identity over $[0,T]$, using
\eqref{eq.17}, and letting $R\to\infty$, we obtain
\begin{align}
\frac12
\int_0^T\sigma(s)^2
\int_{\mathbb R^d}
\beta_\eta''(p_s^\varepsilon)
|\nabla p_s^\varepsilon|^2\,dx\,ds
&\le
\Ent_\eta(p_0^\varepsilon)
-
\Ent_\eta(p_T^\varepsilon)
\nonumber\\
&\quad+
\bigl\|[\nabla\!\cdot f]^-\bigr\|_
{L^1([0,T];L^\infty)},
\label{eq.18}
\end{align}
where
\[
\Ent_\eta(q)
:=
\int_{\mathbb R^d}\beta_\eta(q(x))\,dx.
\]
By the upper bound in \eqref{eq.16}, mass conservation, and
\[
\|p_0^\varepsilon\|_{L^\infty}
\le
\|p_0\|_{L^\infty},
\]
we have
\begin{align*}
\Ent_\eta(p_0^\varepsilon)
&\le
\int_{\mathbb R^d}
p_0^\varepsilon
\log(1+p_0^\varepsilon)\,dx
+
\int_{\mathbb R^d}p_0^\varepsilon\,dx
\le
\log\!\bigl(1+\|p_0\|_{L^\infty}\bigr)+1.
\end{align*}
By the lower bound in \eqref{eq.16},
\[
-\Ent_\eta(p_T^\varepsilon)
\le
-\Ent(p_T^\varepsilon).
\]
The Gaussian entropy comparison gives, for every probability density
$q$ with finite second moment,
\[
\Ent(q)
\ge
-C_d\left(
1+\int_{\mathbb R^d}|x|^2q(x)\,dx
\right).
\]
Hence
\[
-\Ent_\eta(p_T^\varepsilon)
\le
C_d\left(
1+\int_{\mathbb R^d}|x|^2p_T^\varepsilon(x)\,dx
\right)
\le
C_d(1+M).
\]

Here $M<\infty$ is supplied by Step~$7(c')$. That step is justified
using It\^o's formula for the regularised SDE; equivalently, one may
first test against bounded $C^2$ truncations of $|x|^2$, derive an
estimate uniform in the truncation parameter, and then apply monotone
convergence. In particular, the argument does not rely on tightness
alone. Explicitly,
\[
\frac{d}{ds}
\int_{\mathbb R^d}|x|^2p_s^\varepsilon(x)\,dx
=
2\int_{\mathbb R^d}
x\cdot f^\varepsilon(x,s)p_s^\varepsilon(x)\,dx
+
d\,\sigma(s)^2,
\]
and the linear-growth estimate
\[
|f^\varepsilon(x,s)|
\le
2C_f(s)(1+|x|)
\]
gives the uniform second-moment bound by Gr\"onwall's inequality.

Combining these estimates with \eqref{eq.18} yields
\begin{align}
\frac12
\int_0^T\sigma(s)^2
\int_{\mathbb R^d}
\beta_\eta''(p_s^\varepsilon)
|\nabla p_s^\varepsilon|^2\,dx\,ds
&\le
\log\!\bigl(1+\|p_0\|_{L^\infty}\bigr)
+1
+C_d(1+M)
\nonumber\\
&\quad+
\bigl\|[\nabla\!\cdot f]^-\bigr\|_
{L^1([0,T];L^\infty)}.
\label{eq.19}
\end{align}

\medskip
\noindent
\emph{Letting $\eta\downarrow0$.}
Since
\[
\beta_\eta''(r)=\frac1{r+\eta}
\uparrow\frac1r
\qquad\text{as }\eta\downarrow0
\]
for every $r>0$, and $p_s^\varepsilon>0$ for $s>0$, monotone
convergence gives
\[
\int_{\mathbb R^d}
\beta_\eta''(p_s^\varepsilon)
|\nabla p_s^\varepsilon|^2\,dx
\uparrow
I(p_s^\varepsilon).
\]
Using $\sigma(s)^2\ge\epsilon^2$ in \eqref{eq.19}, we therefore obtain
\begin{align}
\frac{\epsilon^2}{2}
\int_0^T I(p_s^\varepsilon)\,ds
&\le
\log\!\bigl(1+\|p_0\|_{L^\infty}\bigr)
+1
+C_d(1+M)
\nonumber\\
&\quad+
\bigl\|[\nabla\!\cdot f]^-\bigr\|_
{L^1([0,T];L^\infty)},
\label{eq.20}
\end{align}
uniformly in $\varepsilon$.

\medskip
\noindent
\emph{Letting $\varepsilon\downarrow0$.}
By Step~$7(e)$,
\[
p^\varepsilon
\longrightarrow p
\quad\text{strongly in }
L^2\bigl([0,T];L^2(K)\bigr)
\]
for every compact $K\subset\mathbb R^d$. Combined with the uniform
tightness from Step~$7(c)$ and the uniform $L^\infty$ bound
\eqref{eq:linfty}, this yields
\[
p^\varepsilon
\longrightarrow p
\quad\text{strongly in }
L^1\bigl([0,T]\times\mathbb R^d\bigr).
\]
Moreover, the global energy estimate from Step~$7(d)$ gives, after
passing to a subsequence if necessary,
\[
\nabla p^\varepsilon
\rightharpoonup
\nabla p
\quad\text{weakly in }
L^2\bigl([0,T]\times\mathbb R^d;\mathbb R^d\bigr).
\]
Define the   function
\[
U(\mu,\xi)
:=
\begin{cases}
\dfrac{|\xi|^2}{\mu},
& \mu>0,\\[1.2ex]
0,
& (\mu,\xi)=(0,0),\\
+\infty,
& \mu=0,\ \xi\neq0.
\end{cases}
\]
Then $U$ is convex and lower semicontinuous on
$[0,\infty)\times\mathbb R^d$. By the standard
lower-semicontinuity theorem for convex integral functionals,
\[
\int_0^T\int_{\mathbb R^d}
U(p_s,\nabla p_s)\,dx\,ds
\le
\liminf_{\varepsilon\downarrow0}
\int_0^T\int_{\mathbb R^d}
U(p_s^\varepsilon,\nabla p_s^\varepsilon)\,dx\,ds.
\]
Equivalently,
\[
\int_0^T I(p_s)\,ds
\le
\liminf_{\varepsilon\downarrow0}
\int_0^T I(p_s^\varepsilon)\,ds.
\]
Passing to the lower limit in \eqref{eq.20} proves \eqref{eq.14}.
Therefore
\[
\int_0^T I(p_s)\,ds<\infty,
\]
which is precisely $(S4')$.
\section{Proof of Proposition \ref{prop:gradient} }\label{app.5.4}
We first use \eqref{eq:confine} to estimate the tail and moments of $\mu$.
For $|x|=r\ge1$ write $x=r\theta$ with
$|\theta|=1$. By \eqref{eq:confine},
\[
  \frac{\dd}{\dd r}V(r\theta)=\big\langle\nabla V(r\theta),\theta\big\rangle
  =\frac1r\big\langle\nabla V(r\theta),r\theta\big\rangle\;\ge\;ar-\frac br ,
\]
and integrating from $1$ to $r$ gives
$V(x)\ge\tfrac a2|x|^2-b\log|x|-c_1$ with
$c_1:=\tfrac a2-\inf_{|y|=1}V(y)<\infty$. Hence, for $|x|\ge1$,
\[ \mu(x)=e^{-2V(x)/\sigma_0^2}\;\le\;e^{2c_1/\sigma_0^2}\,|x|^{2b/\sigma_0^2}\,
  e^{-a|x|^2/\sigma_0^2},
\]
so $\int(1+|x|^2)\,\mu\,\dd x<\infty$; in particular $\mu$ is finite and all polynomial
moments of $\mu$ converge.
Since $p_0=\mu h_0\le C_0\,\mu$, we obtain
$\int|x|^2p_0\,\dd x\le C_0\int|x|^2\mu\,\dd x<\infty$.

By It\^o's formula applied to $|X_s|^2$ 
\[
  \frac{\dd}{\dd s}\E|X_s|^2=-2\,\E\big\langle\nabla V(X_s),X_s\big\rangle+d\sigma_0^2
  \;\le\;-2a\,\E|X_s|^2+2b+d\sigma_0^2 ,
\]
using \eqref{eq:confine} again. Gr\"onwall then gives
$\E|X_s|^2\le e^{-2as}\E|X_0|^2+\tfrac{2b+d\sigma_0^2}{2a}(1-e^{-2as})$, which is
\eqref{eq:confmoment}. Note that  \eqref{eq:confine} is crucial here; see
Remark~\ref{rem:cauchy}.

We now use Cole-Hopf. $\mu$ is invariant, and $h_s:=p_s/\mu$ solves
$\partial_sh=\tfrac{\sigma_0^2}{2}\Delta h-\nabla V\!\cdot\!\nabla h$, i.e.\ $h_s=P_sh_0$
for the Markov semigroup of $\dd X=-\nabla V(X)\dd s+\sigma_0\dd W$. The Markov property gives
$c_0\le h_s\le C_0$, hence $p_s\ge c_0\mu>0$ and
$p_s\le C_0e^{-2\inf V/\sigma_0^2}$; together with the results above this shows (S1).

Because the noise is additive, the Jacobian
$J_s=\nabla_xX^x_s$ solves the pathwise ODE $\dot J_s=-\nabla^2V(X_s)J_s$, so
$\|J_s\|\le e^{\kappa_+s}$; the second variation $K_s=\nabla^2_xX^x_s$ solves
$\dot K_s=-\nabla^2V(X_s)K_s-\nabla^3V(X_s)[J_s,J_s]$ with $K_0=0$, so
$\|K_s\|\le L_3Te^{2\kappa_+T}$. Differentiating $P_sh_0(x)=\E[h_0(X^x_s)]$,
\[
  \|\nabla h_s\|_\infty\le e^{\kappa_+T}\|\nabla h_0\|_\infty,\quad
  \|\nabla^2h_s\|_\infty\le e^{2\kappa_+T}\big(\|\nabla^2h_0\|_\infty
  +L_3T\|\nabla h_0\|_\infty\big).
\]
Since
$\nabla^2\log p_s=-\tfrac{2}{\sigma_0^2}\nabla^2V+\tfrac{\nabla^2h_s}{h_s}
-\tfrac{\nabla h_s\otimes\nabla h_s}{h_s^2}$, we obtain $\opnorm{\nabla^2\log p_s}\le M$,
hence (S2), (S3) and --- $\nabla\log p_s$ being globally $M$-Lipschitz --- also (S4), exactly
as in Proposition~\ref{prop:linear}.

\section{Density estimates used in Section \ref{sec:stability}}
\subsection{Proof of Lemma \ref{lem:kappalinear}} \label{app.linear}
As in Step 1 of Proposition~\ref{prop:linear}, variation of constants gives
$p_s=p_0^{(s)}*\mathcal N(0,\varsigma_s^2I)$ with $p_0^{(s)}(y)=m_s^{-d}p_0(y/m_s)$, the noise
being isotropic because $\sigma$ is scalar and $x$-independent; and $\varsigma_s^2>0$ for
$s>0$ by \textup{(E1)}, so $\varsigma_{\min}^2>0$. Restricting the convolution integral to
$y\in m_sB_\rho$, on which $p_0^{(s)}$ has total mass $\int_{B_\rho}p_0=\mu_\rho$ and
$|x-y|\le R+m_{\max}\rho$,
\[
  p_s(x)\;\ge\;\int_{m_sB_\rho}p_0^{(s)}(y)\,\frac{e^{-|x-y|^2/2\varsigma_s^2}}{(2\pi\varsigma_s^2)^{d/2}}\,\dd y
  \;\ge\;\frac{\mu_\rho}{(2\pi\varsigma_{\max}^2)^{d/2}}\,
  e^{-(R+m_{\max}\rho)^2/2\varsigma_{\min}^2},
\]
which is \eqref{eq:kappalower}. Integrating $1/p_s$ over $[\delta,T]\times B_R$ gives the
bound on $\kappa_{\delta,R}$.

\subsection{Proof of Lemma \ref{lem:kappabdd} }\label{app.nonlinear}
Under \textup{(E1)--(E2)} and $f\in L^\infty\big([0,T]\times\R^d;\R^d\big)$ the operator in \eqref{eq:fp} is uniformly
parabolic in divergence form with bounded measurable coefficients, so Aronson's two-sided
estimates \cite{aronson1967} apply to its fundamental solution $\Gamma(s,x;0,y)$: there are
$c_1,c_2>0$, depending only on $d$, $\epsilon$ and $\|f\|_{L^\infty}$, with
\[
  \Gamma(s,x;0,y)\;\ge\;\frac{c_1}{s^{d/2}}\,\exp\!\Big(-\frac{c_2|x-y|^2}{s}\Big).
\]
Hence, for $s\in[\delta,T]$ and $|x|\le R$, restricting the representation
$p_s(x)=\int\Gamma(s,x;0,y)p_0(y)\dd y$ to $y\in B_\rho$,
\[
  p_s(x)\;\ge\;\mu_\rho\,\inf_{|y|\le\rho}\Gamma(s,x;0,y)
  \;\ge\;\frac{c_1\,\mu_\rho}{T^{d/2}}\,\exp\!\Big(-\frac{c_2(R+\rho)^2}{\delta}\Big)\;>\;0 ,
\]
and integrating $1/p_s$ over $[\delta,T]\times B_R$ gives the claim.

%% file: MimickingODE.bbl
\begin{thebibliography}{10}

\bibitem{Amb2004}
{\sc L.~Ambrosio}, {\em Transport equation and {C}auchy problem for {BV} vector
  fields}, Inventiones Mathematicae, 158 (2004), pp.~227--260.

\bibitem{Ambrosio2008}
{\sc L.~Ambrosio}, {\em Transport equation
  and cauchy problem for non-smooth vector fields}, in Calculus of Variations
  and Non-Linear Partial Differential Equations, vol.~1927 of Lecture Notes in
  Mathematics, Springer, 2008, pp.~1--41.

\bibitem{AGS}
{\sc L.~Ambrosio, N.~Gigli, and G.~Savar{\'e}}, {\em Gradient Flows in Metric
  Spaces and in the Space of Probability Measures}, Lectures in Mathematics ETH
  Z{\"u}rich, Birkh{\"a}user Verlag, Basel, 2nd~ed., 2008.

\bibitem{AT2014}
{\sc L.~Ambrosio and D.~Trevisan}, {\em Well-posedness of {Lagrangian} flows
  and continuity equations in metric measure spaces}, {Analysis \& PDE}, 7
  (2014), pp.~1179--1234.

\bibitem{anderson1982}
{\sc B.~D.~O. Anderson}, {\em Reverse-time diffusion equation models},
  Stochastic Processes and their Applications, 12 (1982), pp.~313--326.

\bibitem{aronson1967}
{\sc D.~G. Aronson}, {\em Bounds for the fundamental solution of a parabolic
  equation.}, {Bull.\ Amer.\ Math.\ Soc.}, 73 (1967), pp.~890--896.

\bibitem{BL1976}
{\sc H.~J. Brascamp and E.~H. Lieb}, {\em On extensions of the
  {B}runn--{M}inkowski and {P}r{\'e}kopa--{L}eindler theorems, including
  inequalities for log concave functions, and with an application to the
  diffusion equation}, Journal of Functional Analysis, 22 (1976), pp.~366--389.

\bibitem{BrigatiPedrotti2025}
{\sc G.~Brigati and M.~Pedrotti}, {\em Heat flow, log-concavity, and
  {Lipschitz} transport maps}, Electronic Communications in Probability, 30
  (2025), pp.~1--12.

\bibitem{chen2025}
{\sc S.~Chen, E.~Vanden-Eijnden, and Y.~Xu}, {\em Lipschitz-guided design of
  interpolation schedules in generative models}, arXiv:2509.01629,  (2025).

\bibitem{CKP2012}
{\sc S.~Cho, P.~Kim, and H.~Park}, {\em Two-sided estimates on dirichlet heat
  kernels for time-dependent parabolic operators with singular drifts in
  $c^{1,\alpha}$-domains}, Journal of Differential Equations, 252 (2012),
  pp.~1101--1145.

\bibitem{CDL2008}
{\sc G.~Crippa and C.~De~Lellis}, {\em Estimates and regularity results for the
  {D}i{P}erna--{L}ions flow}, Journal f{\"u}r die reine und angewandte
  Mathematik, 616 (2008), pp.~15--46.

\bibitem{lellis2008}
{\sc C.~De~Lellis}, {\em {ODE}s with {S}obolev coefficients: the {Eulerian} and
  the {Lagrangian} approach}, Discrete Contin. Dyn. Syst. Ser. S, 1 (2008),
  pp.~405--426.

\bibitem{DL1989}
{\sc R.~J. DiPerna and P.-L. Lions}, {\em Ordinary differential equations,
  transport theory and {S}obolev spaces}, Inventiones Mathematicae, 98 (1989),
  pp.~511--547.

\bibitem{Fig2008}
{\sc A.~Figalli}, {\em Existence and uniqueness of martingale solutions for
  {SDE}s with rough or degenerate coefficients}, Journal of Functional
  Analysis, 254 (2008), pp.~109--153.

\bibitem{FGP2010}
{\sc F.~Flandoli, M.~Gubinelli, and E.~Priola}, {\em Well-posedness of the
  transport equation by stochastic perturbation}, Inventiones Mathematicae, 180
  (2010), pp.~1--53.

\bibitem{gyongy1986}
{\sc I.~Gy{\"o}ngy}, {\em Mimicking the one-dimensional marginal distributions
  of processes having an {It{\^o}} differential}, Probability theory and
  related fields, 71 (1986), pp.~501--516.

\bibitem{hanxu2024}
{\sc Y.~Han, M.~Razaviyayn, and R.~Xu}, {\em Neural network-based score
  estimation in diffusion models: Optimization and generalization}, in
  International Conference on Learning Representations (ICLR), 2024,
  pp.~42520--42558.

\bibitem{haussmann1986}
{\sc U.~G. Haussmann and {\'E}.~Pardoux}, {\em Time reversal of diffusions},
  The Annals of Probability, 14 (1986), pp.~1188--1205.

\bibitem{huang2025convergence}
{\sc D.~Z. Huang, J.~Huang, and Z.~Lin}, {\em Convergence analysis of
  probability flow {ODE} for score-based generative models}, IEEE Transactions
  on Information Theory,  (2025).

\bibitem{iske2026}
{\sc M.~Iske and C.-B. Sch{\"o}nlieb}, {\em Expressivity of bi-{L}ipschitz
  normalizing flows: A score-based diffusion perspective}, arXiv preprint
  arXiv:2605.06172,  (2026).

\bibitem{kim2024}
{\sc D.~Kim, C.-H. Lai, W.~Liao, N.~Murata, Y.~Takida, T.~Uesaka, Y.~He,
  Y.~Mitsufuji, and S.~Ermon}, {\em Consistency trajectory models: Learning
  probability flow ode trajectory of diffusion}, in International Conference on
  Learning Representations, vol.~2024, 2024, pp.~44493--44525.

\bibitem{KR2005}
{\sc N.~V. Krylov and M.~R{\"o}ckner}, {\em Strong solutions of stochastic
  equations with singular time dependent drift}, Probability Theory and Related
  Fields, 131 (2005), pp.~154--196.

\bibitem{LBL2008}
{\sc C.~Le~Bris and P.-L. Lions}, {\em Existence and uniqueness of solutions to
  {F}okker--{P}lanck type equations with irregular coefficients},
  Communications in Partial Differential Equations, 33 (2008), pp.~1272--1317.

\bibitem{LBL2019}
{\sc C.~Le~Bris and P.-L. Lions}, {\em Parabolic Equations
  with Irregular Data and Related Issues}, {De Gruyter}, 2019.

\bibitem{lipman2022}
{\sc Y.~Lipman, R.~T. Chen, H.~Ben-Hamu, M.~Nickel, and M.~Le}, {\em Flow
  matching for generative modeling}, in International Conference on Learning
  Representations (ICLR), 2022.

\bibitem{mooney2025global}
{\sc C.~Mooney, Z.~Wang, J.~Xin, and Y.~Yu}, {\em Global well-posedness and
  convergence analysis of score-based generative models via sharp {Lipschitz}
  estimates}, in International Conference on Learning Representations (ICLR),
  vol.~2025, 2025, pp.~79960--79986.

\bibitem{peyre2026}
{\sc G.~Peyr\'e}, {\em Optimal and diffusion transports in machine learning},
  in Proceedings of the{ International Congress of Mathematicians} 2026,
  vol.~7: Invited Lectures (Sections 15–20), pp.~110--128.

\bibitem{shen2026}
{\sc Y.~Shen, Y.~Xi, J.~Ren, X.~Zhang, Y.~Wang, B.~Chen, Z.~Zheng, N.~Liang,
  C.~Wang, A.~Cai, et~al.}, {\em {ODE}-driven deterministic diffusion posterior
  sampling for sparse-view ct image reconstruction}, IEEE Transactions on
  Radiation and Plasma Medical Sciences,  (2026).

\bibitem{SohlDickstein2015}
{\sc J.~Sohl-Dickstein, E.~A. Weiss, N.~Maheswaranathan, and S.~Ganguli}, {\em
  Deep unsupervised learning using nonequilibrium thermodynamics}, in
  Proceedings of the 32nd International Conference on Machine Learning, F.~Bach
  and D.~Blei, eds., vol.~37 of Proceedings of Machine Learning Research,
  Lille, France, 2015, PMLR, pp.~2256--2265.

\bibitem{Song2021mle}
{\sc Y.~Song, C.~Durkan, I.~Murray, and S.~Ermon}, {\em Maximum likelihood
  training of score-based diffusion models}, in Advances in Neural Information
  Processing Systems, A.~Beygelzimer, Y.~Dauphin, P.~Liang, and J.~W. Vaughan,
  eds., vol.~34, 2021.

\bibitem{Song2021}
{\sc Y.~Song, J.~Sohl-Dickstein, D.~P. Kingma, A.~Kumar, S.~Ermon, and
  B.~Poole}, {\em Score-based generative modeling through stochastic
  differential equations}, in International Conference on Learning
  Representations (ICLR), 2021.

\bibitem{Stephanovitch2025ScoreRegularity}
{\sc A.~St{\'e}phanovitch}, {\em Regularity of the score function in generative
  models}, arXiv:2506.19559,  (2025).

\bibitem{Stephanovitch2026Lipschitz}
{\sc A.~St{\'e}phanovitch},  {\em Lipschitz regularity
  in flow matching and diffusion models: Sharp sampling rates and functional
  inequalities}, arXiv:2604.06065,  (2026).

\bibitem{Ver1981}
{\sc A.~J. Veretennikov}, {\em On strong solutions and explicit formulas for
  solutions of stochastic integral equations}, Mathematics of the USSR-Sbornik,
  39 (1981), pp.~387--403.

\bibitem{yang2023}
{\sc L.~Yang, Z.~Zhang, Y.~Song, S.~Hong, R.~Xu, Y.~Zhao, W.~Zhang, B.~Cui, and
  M.-H. Yang}, {\em Diffusion models: A comprehensive survey of methods and
  applications}, ACM computing surveys, 56 (2023), pp.~1--39.

\end{thebibliography}
